\documentclass{amsart}

\usepackage{cite}

\usepackage{verbatim}
\usepackage{amssymb}
\usepackage{enumerate}
\usepackage[active]{srcltx}
\usepackage{tikz}
\usepackage[utf8]{inputenc}
 \usepackage[usenames,dvipsnames]{pstricks}
 \usepackage{epsfig}
\usepackage{pst-grad}
\usepackage{hyperref}

\usepackage{float}
\usepackage[margin=3cm]{geometry}   

\newcommand{\cl}{\textrm{cl}}
\newcommand{\onto}{\twoheadrightarrow}

\newcommand{\tri}{\vartriangleleft}
\newcommand{\vsubs}{\rotatebox[origin=c]{90}{$\subset$}}
\newtheorem*{claimn}{Claim}

\newtheorem{theorem}{Theorem}[section]
\newtheorem{definition}{Definition}[section] 
\newtheorem{lemma}[theorem]{Lemma}     
\newtheorem{corollary}[theorem]{Corollary}

\newtheorem{fact}[theorem]{Fact}
\newtheorem{dfn}[theorem]{Definition}

\newtheorem{obs}[theorem]{Observation}

\newtheorem{claim}{Claim}[theorem]

\newcommand{\uhr}{\upharpoonright}

\def\myheads#1;#2;{
\pagestyle{myheadings}
\markboth{{\sc\hfill #1\hfill\protect\makebox[0cm][r]{\rm\today}}}
{{\sc\protect\makebox[0cm][l]{\rm\today}\hfill #2\hfill}}
}

\newif\ifdeveloping

\developingtrue

\ifdeveloping
\fi

\newif\ifcommented
\commentedtrue

\newcommand{\comm}[1]{}

\ifcommented
\renewcommand{\comm}[1]{
\fbox{\fbox{\begin{minipage}{300pt}#1\end{minipage}}
}}

\fi

\def\br#1;#2;{\bigl[ {#1} \bigr]^ {#2} }

\newcommand{\mc}[1]{\mathcal{#1}}
\newcommand{\mbb}[1]{\mathbb{#1}}

\newcommand{\mf}[1]{\mathfrak{#1}}
\newcommand{\uhp}{\upharpoonright}

\newcommand{\omg}{{\omega_1}}

\newcommand{\cfc}{{\chi}_{CF}}

\newcommand{\lex}{<_{\scalebox{0.8}{\textrm{lex}}}}

\newcommand{\gr}{G=(V,E)}
\newcommand{\setm}{\setminus}

\newcommand{\subs}{\subset}

\newcommand{\cf}{\textrm{cf}}

\newcommand{\dom}{\textrm{dom}}
\newcommand{\ran}{\textrm{ran}}
\newcommand{\oo}{\omega}

\def\<{\left\langle}
\def\>{\right\rangle}
\def\br#1;#2;{\bigl[ {#1} \bigr]^ {#2} }

\definecolor{cthm}{rgb}{0,0,0.7} 
\definecolor{cdef}{rgb}{1,0,0} 
\definecolor{midg}{rgb}{0,0.7,0} 
\definecolor{cemph}{rgb}{0,0,1} 
\definecolor{depg}{rgb}{0,0.5,0} 

\author[D. T. Soukup]{D\'aniel T. Soukup}

  \address[D.T. Soukup]{Universität Wien,
 Kurt Gödel Research Center for Mathematical Logic, Austria}
  \email[Corresponding author]{daniel.soukup@univie.ac.at}
\urladdr{http://www.logic.univie.ac.at/$\sim  $soukupd73/}

\author[L. Soukup]{Lajos Soukup}

\address[L. Soukup]{Alfr{\'e}d R{\'e}nyi Institute of Mathematics, Hungarian Academy of Sciences, Budapest, Hungary  }
\email{soukup@renyi.hu}
\urladdr{http://www.renyi.hu/$\sim  $soukup}

\title{Infinite combinatorics plain and simple}
\date{\today}

\begin{document}
\maketitle

\begin{abstract}

We explore a general method based on trees of elementary submodels in order to present highly simplified proofs to numerous results in infinite combinatorics. 
While countable elementary submodels have been employed in such settings already, we significantly broaden this framework by developing the corresponding technique for countably closed models of size continuum. The applications range from various theorems on paradoxical decompositions of the plane, to coloring sparse set systems, results on graph chromatic number and  constructions from point-set topology. Our main purpose is to demonstrate the ease and wide applicability of this method in a form accessible to anyone with a basic background in set theory and logic.

\end{abstract}

\tableofcontents


\section{Introduction}

Solutions to combinatorial problems often follow the same head-on approach: enumerate certain objectives and then inductively meet these goals. Imagine that you are asked to color the points of a topological space with red and blue so that both colors appear on \emph{any} copy of the Cantor-space in $X$. So, one lists the Cantor-subspaces and inductively declares one point red and one point blue from each; this idea, due to Bernstein, works perfectly if $X$ is small i.e. size at most the continuum. However, for larger spaces, we might run into the following problem: after continuum many steps, we could have accidentally covered some Cantor-subspace with red points only. So, how can we avoid such a roadblock?

The methods to meet the goals in the above simple solution scheme vary from problem to problem, however the techniques for \emph{finding the right enumeration} of infinitely or uncountably many objectives frequently involve the same idea. In particular, a recurring feature is to write our set of objectives $\mc X$ as a union of smaller pieces $\langle \mc X_\alpha:\alpha<\kappa\rangle$ so that each $\mc X_\alpha$ resembles the original structure $\mc X$. This is what we refer to as a \emph{filtration}. In various situations, we need the filtration to consist of countable sets; in others, we require that $\mc X_\alpha\subseteq \mc X_\beta$ for $\alpha<\beta<\kappa$. In the modern literature, the sequence $\langle \mc X_\alpha:\alpha<\kappa\rangle$ is more than often defined by intersecting $\mc X$ with an increasing chain of countable elementary submodels; in turn, elementarity  allows properties of $\mc X$ to reflect. 

The introduction of elementary submodels to solving combinatorial problems was truly revolutionary. It provided deeper insight and simplified proofs to otherwise technical results. 
Nonetheless, note that any set $\mc X$ which is covered by an increasing family of countable sets must have size at most $\aleph_1$, a rather serious limitation even when considering problems arising from the reals. Indeed, this is one of the reasons that the assumption $2^{\aleph_0}=\aleph_1$, i.e. the Continuum Hypothesis, is so ubiquitous when dealing with uncountable structures.

On the other hand, several results which seemingly require the use of CH can actually be proved without any extra assumptions. So now the question is, how can we define reasonable filtrations by countable sets to cover structures of size bigger than $\aleph_1$? It turns out that one can relax the assumption of the filtration being increasing in a way which still allows many of our usual arguments for chains to go through. This is done by using a \emph{tree of elementary submodels} rather than chains, an idea which we believe originally appeared in a paper of R. O. Davies \cite{davies} in the 1960s.

Our first goal will be to present Davies' technique in detail using his original result and some other simple and, in our opinion, entertaining new applications. However, this will not help in answering the question from the first paragraph. So, we develop the corresponding technique based on \emph{countably closed elementary submodels of size continuum}. This allows us to apply Davies' idea in a much broader context, in particular, to finish Bernstein's argument for coloring topological spaces. As a general theme, we present simple results answering natural questions from combinatorial set theory; in many cases, our new proofs replace intimidating and technically involved arguments from the literature. We hope to do all this while keeping the paper self contained and, more importantly, accessible to anyone with a basic background in set theory and logic.

 Despite its potential, Davies' method is far from common knowledge even today, though we hope to contribute to changing this. In any case, we are not the first to realize the importance of this method: S. Jackson and D. Mauldin used the same technique to solve the famous Steinhaus tiling problem \cite{jackson}; also, such filtrations were successfully applied and popularized by D. Milovich \cite{milovich,milovichFM1,milovichFM2, milovichFM3} under the name of (long) $\omg$-approximation sequences. In particular, the authors learned about this beautiful technique from Milovich so we owe him a lot.

\smallskip

The structure of our paper is the following: we start by looking at a theorem of W. Sierpinski to provide a mild introduction to  elementary submodels in Section \ref{sec:case}. Then,  in Section \ref{sec:dtree}, we define our main objects of study: the trees of elementary submodels we call Davies-trees; in turn, we present Davies' original application. We continue with four further (independent) applications of varying difficulty: in Sections \ref{sec:disjointness} and \ref{sec:cffree}, we look at almost disjoint families of countable sets and conflict-free colorings. Next, we present a fascinating result of P. Komj\'ath and J. Schmerl in Section \ref{sec:clouds}: lets say that $A\subs \mbb R^2$ is a \emph{cloud} iff every line $\ell$ through a fixed point $a$ intersects $A$ in a finite set only. Now, how many clouds can cover the plane? We certainly need at least two, but the big surprise is the following: the cardinal arithmetic assumption $2^{\aleph_0}\leq \aleph_m$ is equivalent to $m+2$ clouds covering $\mbb R^2$. In our final application of regular Davies-trees, in Section \ref{sec:chrom}, we will show that any graph with uncountable chromatic number necessarily contains highly connected subgraphs with uncountable chromatic number. 

The second part of our paper deals with a new version of Davies' idea, which we dubbed \emph{high Davies-trees}. High Davies-trees will be built from countably closed models $M$ i.e. $A\in M$ whenever $A$ is a countable subset of $M$. This extra assumption implies that $M$ has size at least $\mf c$ but can be extremely useful when considering problems which involve going through all countable subsets of a structure $\mc X$. So our goal will be to find a nice filtration $\langle \mc X_\alpha:\alpha<\kappa\rangle$ of a structure $\mc X$ \emph{simultaneously} with its countable subsets; while $\mc X$ can be arbitrary large, each $\mc X_\alpha$ will have size continuum only.

In Section \ref{sec:high}, we introduce high Davies-trees precisely. In contrast to regular Davies-trees which exist in ZFC, we do need extra set-theoretic assumptions to construct high Davies-trees: a weak version of GCH and Jensen's square principle. This construction is carried out only in Section \ref{sec:appendix} as we would like to focus on applications first. We mention that the results of Section \ref{sec:morechrom}, \ref{sec:freese} and \ref{sec:top} show that high Davies-trees might not exist in some models of GCH and so extra set theoretic assumptions are necessary to construct these objects.

Now, our main point is that high Davies-trees allow us to provide clear presentation to results which originally had longer and sometimes fairly intimidating proofs. In particular, in Section \ref{sec:B}, we prove a strong form of Bernstein's theorem: any Hausdorff topological space $X$ can be colored by continuum many colors so that each color appears on any copy of the Cantor-space in $X$. Next, in Section \ref{sec:saturated}, we present a new proof that there are saturated almost disjoint families in $[\kappa]^\oo$ for any cardinal $\kappa$. In our last application in Section \ref{sec:top}, we show how to easily construct nice locally countable, countably compact topological spaces of arbitrary large size. All Sections through \ref{sec:morechrom} to \ref{sec:top} can be read independently after Section  \ref{sec:high}.

Compared to the original statement of the above results, we only require the existence of appropriate high Davies-trees instead of various $V=L$ --like assumptions. This supports our belief that the existence of high Davies-trees can serve as a practical new axiom or blackbox which captures certain useful, but otherwise technical, combinatorial consequences of $V=L$ in an easily applicable form. The obvious upside being that anyone can apply Davies-trees without familiarity with the constructible universe or square principles.

\smallskip

Ultimately, Davies-trees will not solve open problems magically; indeed, we mention still unanswered questions in almost each section. But hopefully we managed to demonstrate that Davies-trees do provide an invaluable tool for understanding the role of CH and $V=L$ in certain results, and for a revealing, modern presentation of otherwise technically demanding arguments.

\smallskip

\textbf{Disclaimer.} We need to point out that it is not our intention to give a proper introduction to elementary submodels, with all the prerequisites in logic and set theory, since this has been done in many places already. If this is the first time the kind reader encounters elementary submodels, we recommend the following sources: Chapter III.8 in K. Kunen's book \cite{kunen} for classical applications and W. Just, M. Weese \cite{justel} for a more lengthy treatment; A. Dow \cite{dowel} and S. Geschke \cite{gesel} survey applications of elementary submodels in topology; a paper by the second author \cite{soukel} gives all the required background in logic with a focus on graph theory. Reading (the first few pages of) any of the above articles will give the additional background for the following sections to anyone with a set theory course already under his or her belt. Nonetheless, an informal introduction is included in Section \ref{sec:case}.

\smallskip

 We use standard notations following \cite{kunen}. ZFC denotes the usual Zermelo-Fraenkel axioms of set theory together with the Axiom of Choice. We use $\mf c$ to denote  $2^{\aleph_0}$, the cardinality of $\mbb R$. CH denotes the Continuum Hypothesis i.e. $\mf c=\aleph_1$. We say that GCH holds (i.e. the generalized Continuum Hypothesis) if $2^\lambda=\lambda^+$ for all infinite $\lambda$.

\section{A case study}\label{sec:case}

We begin by examining a result of Sierpinski from the 1930s to demonstrate the use of filtrations. W. Sierpinski produced a myriad of results \cite{sierp_book} relating CH to  various properties of the reals. Some properties were proved to be equivalent to CH; for others, like the theorem below, Sierpinski could not decide if an actual equivalence holds.

\begin{theorem}[\cite{sierp}]\label{sierp} CH implies that $\mbb R^2$ can be covered by countably many rotated graphs of functions i.e. given distinct lines $\ell_i$ for $i<\oo$ through the origin there are sets $A_i$ so that $\mbb R^2=\bigcup \{A_i:i<\oo\}$ and  $A_i$ meets each line perpendicular to $\ell_i$ in at most one point.
\end{theorem}

\begin{proof}
 Fix distinct lines $\ell_i$ for $i<\oo$ through the origin. Our goal is to find sets $A_i$ so that $\mbb R^2=\bigcup \{A_i:i<\oo\}$ and if $\ell\perp \ell_i$ (that is, $\ell$ and $\ell_i$ are perpendicular) then $|A_i\cap \ell|\leq 1$.
 
 If our only goal would be to cover a countable subset $R_0=\{r_0,r_1 \dots\}$ of $\mbb R^2$ first then we can simply let $A_i=\{r_i\}$. What if we wish to extend this particular assignment to a cover of a larger set $R_1\supseteq R_0$? Given some $r\in R_1\setm R_0$, the only obstacle of putting $r$ into $A_i$ is that the line $\ell=\ell(r,i)$ through $r$ perpendicular to $\ell_i$ meets $A_i$ already; we will say that $i$ is bad for $r$ (see Figure \ref{fig:badpoints}).

 \begin{figure}[H]
  \centering
	\psscalebox{0.7 0.7}
{
\begin{pspicture}(0,-3.8764062)(8.388282,3.8564062)
\psline[linewidth=0.04cm](1.42,3.0764062)(4.6,-3.6435938)
\psline[linewidth=0.04cm](0.0,-0.44359374)(7.64,-1.8635937)
\psdots[dotsize=0.16](3.38,-1.0835937)
\psline[linewidth=0.04cm,linestyle=dashed,dash=0.16cm 0.16cm](5.72,3.8364062)(4.9,-1.9835937)
\psline[linewidth=0.04cm,linestyle=dashed,dash=0.16cm 0.16cm](1.6,-0.32359374)(7.52,2.9764063)
\psdots[dotsize=0.16](4.2,1.0764062)
\psdots[dotsize=0.16](5.22,0.15640625)
\psdots[dotsize=0.24](5.46,1.8164062)
\usefont{T1}{ptm}{b}{it}
\rput(7.7915626,-2.4135938){$\ell_i$}
\usefont{T1}{ptm}{m}{n}
\rput(3.5459373,-3.6535938){$\ell_j$}
\usefont{T1}{ptm}{m}{n}
\rput(6.255937,1.7664063){$r$}
\usefont{T1}{ptm}{m}{n}
\rput(4.495938,1.8264062){$r_j$}
\usefont{T1}{ptm}{m}{n}
\rput(5.9359374,0.56640625){$r_i$}
\psbezier[linewidth=0.04](2.82,2.9564064)(2.82,2.1564062)(3.5298266,2.3493474)(4.02,1.4164063)(4.5101733,0.48346514)(4.22,0.03640625)(4.14,-0.54359376)
\psbezier[linewidth=0.04](3.62,0.01640625)(4.02,-0.18359375)(4.372718,0.037566073)(5.12,0.15640625)(5.8672824,0.27524644)(6.4,-0.14359374)(7.04,-1.0035938)
\usefont{T1}{ptm}{m}{n}
\rput(7.831406,-0.79859376){$A_i$}
\usefont{T1}{ptm}{m}{n}
\rput(3.5614061,3.0814064){$A_j$}
\end{pspicture} 
}  \caption{$i$ and $j$ are both bad for $r$}
  \label{fig:badpoints}
\end{figure}
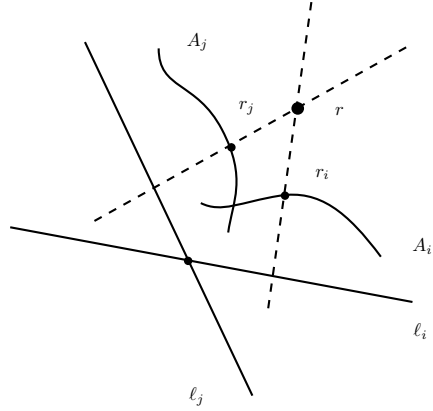

 If all $i<\oo$ are bad for $r$ simultaneously then we are in trouble since there is no way we can extend this cover to include $r$.
 
 However, note that if $i\neq j$ are both bad for $r$ then $r$ can be constructed from the lines $\ell_i,\ell_j$ and the points $r_i,r_j$ (see Figure \ref{fig:badpoints}). So constructible points pose the only obstacle for defining further extensions. Hence, in the first step, we should choose $R_0$ so that it is closed under such constructions. It is easy to see that any set $R$ is contained in a set $R^*$ of size $|R|+\oo$ which is closed in the following sense: if $x,y\in R$ and $i\neq j<\oo$ then the intersection of $\ell(x,i)$ and $\ell(y,j)$ is also in $R$.
 
 Now the complete proof in detail: using CH, list $\mbb R^2$ as $\{r_\alpha:\alpha<\omg\}$ and construct an increasing sequence $R_0\subset R_1\subset \dots \subseteq R_\alpha$ for $\alpha<\omg$ so that $R_0=(\{r_0\})^*$, $R_{\alpha+1}=(R_\alpha\cup \{r_\alpha\})^*$ and $R_{\beta}=\bigcup\{R_\alpha:\alpha<\beta\}$ for $\beta$ limit. Note that each $R_\alpha$ is countable and closed in the above sense. Furthermore, $\bigcup\{R_\alpha:\alpha<\omg\}=\mbb R^2$.
 
 Next, we define the sets $\{A_i:i<\oo\}$ by first distributing $R_0$ trivially as before and then inductively distributing the points in the differences $R_{\alpha+1}\setm R_\alpha$. $R_{\alpha+1}\setm R_\alpha$ is countable so we can list it as $\{t_n:n<\oo\}$. Given that the points of $R_\alpha$ are assigned to the $A_i$'s already, we will put $t_n$ into $A_{2n}$ or $A_{2n+1}$. Why is this possible? Because if both $2n$ and $2n+1$ are bad for $t_n$ then $t_n$ is constructible from points in $R_\alpha$ and hence $t_n\in R_\alpha$ which is a contradiction. This finishes the induction and hence the proof of the theorem.
 \end{proof}

 One can think of the sequence $\{R_\alpha:\alpha<\omg\}$ as a \emph{scheduling of objectives}, where in our current situation an objective is a point in $\mbb R^2$ that needs to be covered. We showed that we can easily cover new points given that the sets $R_\alpha$ were closed under constructibility.
 
 Is there any way of knowing, in advance, under what operations exactly our sets need to be closed? It actually does not matter. Indeed, one can define the closure operation using any countably many functions (and not just this single geometric constructibility) and the closure $R^*$ of $R$ still has the same size as $R$ (given that $R$ is infinite). So we can start by saying that $\{R_\alpha:\alpha<\omg\}$ is a filtration of $\mbb R^2$ closed under \emph{any conceivable operation} and, when time comes in the proof, we will use the fact that each $R_\alpha$ is closed under the particular operations that we are interested in. 
 
 Let us make the above argument more precise, and introduce elementary submodels, because the phrase \emph{conceivable operation} is far from satisfactory. 

 Let $V$ denote the set theoretic universe and let $\theta$ be a cardinal. Let $H(\theta)$ denote the collection of sets of hereditary cardinality $<\theta$. $H(\theta)$ is actually a set which highly resembles $V$. Indeed, all of ZFC except the power set axiom is satisfied by the model $(H(\theta),\in)$; moreover, if $\theta$ is large enough then we can take powers of small sets while remaining in $H(\theta)$. So, choosing $\theta$ large enough ensures that any argument, with limited iterations of the power set operation, can be carried out in $H(\theta)$ instead of the proper model of ZFC.
 
 \begin{dfn}\label{eldef} We say that a subset $M$ of $H(\theta)$ is an elementary submodel iff for any first order formula $\phi$ with parameters from $M$ is true in $(M,\in)$ iff it is true in $(H(\theta),\in)$. We write $M\prec H(\theta)$ in this case. 
 \end{dfn}

 The reason to work with the models $(H(\theta),\in)$ is that we cannot express ``$(M,\in)\prec (V,\in)$'' in first-order set theory (see \cite[Theorem 12.7]{jech}), but resorting to $M\prec H(\theta)$ avoids the use of second-order set theory. Furthermore,  for any countable set $A\subseteq H(\theta)$, there is a countable elementary submodel $M$ of $H(\theta)$ with $A\subseteq M$ by the downward L\"owenheim-Skolem theorem \cite[Theorem I.15.10]{kunen}. 

Now, a lot of objects are automatically included in any nonempty $M\prec H(\theta)$: all the natural numbers, $\oo$, $\omg$ or $\mbb R$. Also, note that $M\cap \mbb R^2$ is closed under any operation which is defined by a first order formula with parameters from $M$.

We need to mention that countable elementary submodels are far from transitive i.e. $A\in M$ does not imply $A\subseteq M$ in general. However, the following is true:
 
 \begin{fact}\label{elfact} Suppose that $M$ is a countable elementary submodel of $H(\theta)$ and $A\in M$. If $A$ is countable then $A\subseteq M$ or equivalently, if $A\setm M$ is nonempty then $A$ is uncountable.
\end{fact}

We present the proof to familiarize the reader with Definition \ref{eldef}.

\begin{proof} Suppose that $A\in M$ and $A$ is countable. So $H(\theta)\models |A|\leq \omega$. This means that $H(\theta)\models \phi(A)$ where $\phi(x)$ is the sentence ``there is a surjective $f:\oo\to x$''. In turn, $M\models \phi(A)$ so we can find $f\in M$  such that $\dom(f)=\oo$ and $\ran(f)=A$. We conclude that $f(n)\in M$ since $n\in M$ for all natural numbers $n\in \oo$. Hence $A=\ran(f)\subseteq M$.
\end{proof}

Now, how do elementary submodels connect to the closed sets in the proof of Theorem \ref{sierp}? Set $\theta=\mathfrak c^+$ and take a countable  $M\prec H(\mathfrak c^+)$ which contains our fixed set of lines $\{\ell_i:i<\oo\}$. We let $R=M\cap \mbb R^2$ and claim that $R$ is closed under constructibility. Indeed, if $i<\oo$ and $x\in R\setm \ell_i$ then the fact that ``$\ell$ is a line through $x$ orthogonal to $\ell_i$'' is expressible by a first order formula with parameters in $M$. Hence $M$ contains as an element a (more precisely, the unique) line $\ell=\ell(x,i)$ through $x$ that is orthogonal to $\ell_i$. So the intersection $\ell(x,i)\cap \ell(y,j)$ is an element of $M$ if $x,y\in R$. Since  $\ell(x,i)\cap \ell(y,j)$ is a singleton, Fact \ref{elfact} implies that $\ell(x,i)\cap \ell(y,j)\subseteq R$.

In order to build the filtration $\{R_\alpha:\alpha<\omg\}$, we take a continuous, increasing sequence of countable elementary submodels $M_\alpha\prec H(\mathfrak c^+)$ and set $R_\alpha=M_\alpha\cap \mbb R^2$. If we make sure that $r_\alpha\in M_{\alpha+1}$ then $\{R_\alpha:\alpha<\omg\}$ covers $\mbb R^2$.\\


\section{Chains versus trees of elementary submodels}\label{sec:dtree}


In the proof of Theorem \ref{sierp}, CH was imperative and Sierpinski already asked if CH was in fact necessary to show Theorem \ref{sierp}. 
The answer came from R. O. Davies \cite{davies}  who proved that neither CH nor any other extra assumption is necessary beyond ZFC.

We will present his proof now as a way of introducing the main notion of our paper: Davies-trees. Davies-trees will be special sequences of countable elementary submodels $\langle M_\alpha:\alpha<\kappa\rangle$ covering a structure of size $\kappa$. Recall that if $\kappa>\omg$ (e.g. we wish to cover $\mbb R^2$ while CH fails) then we cannot expect the $M_\alpha$'s to be increasing so what special property can help us out?   



The simple idea is that we can always cover a structure of size $\kappa$ with a continuous chain of elementary submodels of size $<\kappa$ so lets see what happens if we repeat this process and cover each elementary submodel again with chains of smaller submodels, and those submodels with chains of even smaller submodels and so on $\dots$ The following result is a simple version of \cite[Lemma 3.17]{milovich}:

\begin{theorem}\label{main} Suppose that $\kappa$ is  cardinal, $x$ is a set. Then there is a large enough cardinal $\theta$ and a sequence of $\langle M_\alpha:\alpha<\kappa\rangle$ of  elementary submodels of $H(\theta)$ so that
\begin{enumerate}[(I)]
	\item $|M_\alpha|=\oo$ and $x\in M_\alpha$ for all $\alpha<\kappa$,
	\item $\kappa\subs \bigcup_{\alpha<\kappa} M_\alpha$, and
	\item\label{ww} for every $\beta<\kappa$ there is $m_\beta\in \mbb N$ and models $N_{\beta,j}\prec H(\theta)$ such that $ x\in N_{\beta,j}$ for $j<m_\beta$ and $$\bigcup\{M_\alpha:\alpha<\beta\}=\bigcup\{N_{\beta,j}:j<m_\beta\}.$$
\end{enumerate}
\end{theorem}

We will refer to such a sequence of models as a \emph{Davies-tree for $\kappa$ over $x$} in the future (and we will see shortly why they are called trees). The cardinal $\kappa$ will denote the size of the structures that we deal with (e.g. the size of $\mbb R^2$) while the set $x$ contains the objects relevant to the particular situation (e.g. a set of lines).

Note that if the sequence $\langle M_\alpha:\alpha<\kappa\rangle$ is increasing then $\bigcup\{M_\alpha:\alpha<\beta\}$ is also an elementary submodel of $H(\theta)$ for each $\beta<\kappa$; as we said already, there is no way to cover a set of size bigger than $\omg$ with an increasing chain of \emph{countable} sets. Theorem \ref{main} says that we can cover by countable elementary submodels and almost maintain the property that the initial segments $\bigcup\{M_\alpha:\alpha<\beta\}$ are submodels. Indeed, each initial segment is the union of \emph{finitely many} submodels by condition (\ref{ww}) while these models still contain everything relevant (denoted by $x$ above) as well.

\begin{proof}[Proof of Theorem \ref{main}] Let $\theta$ be large enough so that $\kappa,x\in H(\theta)$.
 We recursively construct a tree $T$ of finite sequences of ordinals and elementary submodels $M(a)$ for $a\in T$. Let $\emptyset\in T$ and let $M(\emptyset)$ be an elementary submodel of size $\kappa$ so that  
\begin{itemize}
	\item $x\in M(\emptyset)$,
	\item ${\kappa \subs M(\emptyset)}$.
\end{itemize}
Suppose that we defined a tree $T'$ and corresponding models $M(a)$ for $a\in T'$. Fix $a\in T'$ and suppose that $M(a)$ is \emph{uncountable}. Find a continuous, increasing sequence of elementary submodels $\<M(a^\frown\xi)\>_{\xi<\zeta}$ so that
\begin{itemize}
	\item $x\in  M(a^\frown\xi)$ for all $\xi<\zeta$,
	\item $M(a^\frown\xi)$ has size strictly less than $M(a)$, and 
	\item $M(a)=\bigcup\{M(a^\frown\xi):\xi<\zeta\}$.
\end{itemize}
 We extend $T'$ with $\{a^\frown \xi:\xi<\zeta\}$ and iterate this procedure to get $T$.


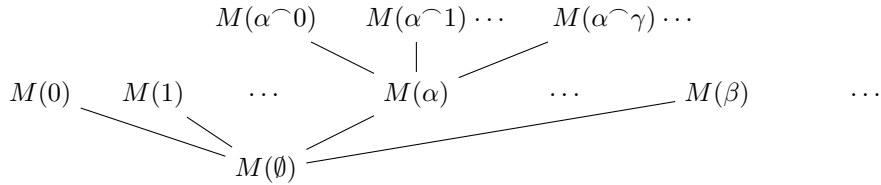
\vspace{0.4 cm}
\hspace{1 cm}
\begin{tikzpicture}[]

\draw (0,0)   node  (me)  {$M(\emptyset)$};

\draw (-3,1)   node   (m0) {$M(0)$};
\draw (-1.5,1)   node (m1)   {$M(1)$};
\draw (0,1)   node    {$\dots$};
\draw (2,1)   node (ma)   {$M(\alpha)$};
\draw (4,1)   node    {$\dots$};
\draw (6,1)   node (mb)   {$M(\beta)$};
\draw (8,1)   node    {$\dots$};

\draw  (me) -- (m0);
\draw  (me) -- (m1);
\draw  (me) -- (ma);
\draw  (me) -- (mb);

\draw (0,2)   node   (ma0) {$M({\alpha}^\frown 0)$};
\draw (2,2)   node (ma1)   {$M({\alpha}^\frown 1)$};
\draw (3,2)   node    {$\dots$};
\draw (4.5,2)   node (maa)   {$M({\alpha}^\frown\gamma)$};
\draw (5.5,2)   node    {$\dots$};

\draw  (ma) -- (ma0);
\draw  (ma) -- (ma1);
\draw  (ma) -- (maa);

\end{tikzpicture}


\vspace{0.2 cm}

It is easy to see that this process produces a downwards closed subtree $T$ of $\textrm{Ord}^{<\oo}$ and if $a\in T$ is a terminal node  then $M(a)$ is countable. Let us well order $\{M(a): a\in T$ is a terminal node$\}$ by the lexicographical ordering ${\lex}$. 

First, note that the order type of  ${\lex}$ is $\kappa$ since $\{M(a): a\in T$ is a terminal node$\}$ has size $\kappa$ and each $M(a)$ has $<\kappa$ many ${\lex}$-predecessors.

We wish to show that if $b\in T$ is terminal then $\bigcup\{M(a):a{\lex}b, a\in T$ is a terminal node$\}$ is the union of finitely many submodels containing $x$. Suppose that $|b|=m\in \mbb N$ and write $$N_{b,j}=\bigcup\{M((b\uhp {j-1})^\frown \xi):\xi<b(j-1)\}$$ for $j=1\dots m$. It is clear that $N_{b,j}$ is an elementary submodel as a union of an increasing chain. Also, if $a{\lex}b$ then $M(a)\subs N_{b,j}$ must hold where $j=\min\{i\leq n: a(i)\neq b(i)\}$. So $$\bigcup\{M(a):a\lex b\textmd{ is terminal}\}=\bigcup\{N_{b,j}:j<m\}$$ as desired.

\end{proof}

\textbf{Remarks.} Note that this proof shows that if $\kappa= \aleph_n$ then every initial segment in the lexicographical ordering is the union of $n$ elementary submodels (the tree $T$ has height $n$). 

In the future, when working with a sequence of elementary submodels $\langle M_\alpha:\alpha<\kappa\rangle$, we use the notation $$ M_{<\beta}=\bigcup\{M_\alpha:\alpha<\beta\}$$ for $\beta<\kappa$.
\vspace{0.2 cm}

Let us outline now Davies' result which removes CH from Theorem \ref{sierp}:

\begin{theorem}$\mbb R^2$ can be covered by countably many rotated graphs of functions. 
\end{theorem}
\begin{proof} Fix distinct lines $\ell_i$ for $i<\oo$ through the origin. As before, our goal is to find sets $A_i$ so that $\mbb R^2=\bigcup \{A_i:i<\oo\}$ and if $\ell\perp \ell_i$ then $|A_i\cap \ell|\leq 1$.

Let $\kappa=\mf c$ and take a Davies-tree $\langle M_\alpha:\alpha<\kappa\rangle$ for $\kappa$ over $\{\kappa, \mbb R^2, r,\ell_i:i<\oo\}$ where $r:\kappa\to \mbb R^2$ is onto. So, if $\xi\in \kappa\cap M_\alpha$ then $r(\xi)\in \mbb R^2\cap M_\alpha$. In turn, $\mbb R^2\subseteq \bigcup \{M_\alpha:\alpha<\kappa\}$.

By induction on $\beta<\kappa$, we will distribute the points in $\mbb R^2\cap M_{<\beta}$ among the sets $A_i$ while making sure that if $\ell\perp \ell_i$ then $|A_i\cap \ell|\leq 1$. In a general step, we list the countable set $\mbb R^2\cap M_\beta\setm M_{<\beta}$ as $\{t_n:n<\oo\}$. Suppose we were able to put $t_k$ into $A_{i_k}$ for $k<n$ and we consider $t_n$. 

Recall that $M_{<\beta}$ can be written as $\bigcup\{N_{\beta,j}:j<m_\beta\}$ for some finite $m_\beta$ where each $N_{\beta,j}$ is an elementary submodel containing $\{\kappa, \mbb R^2, r,\ell_i:i<\oo\}$. In turn, $\mbb R^2\cap M_{<\beta}$ is the union of $m_\beta$ many sets which are closed under constructibility using the lines $\{\ell_i:i<\oo\}$. This means that there could be at most $m_\beta$ many $i\in \oo\setm \{i_k:k<n\}$ which is bad for $t_n$ i.e. such $i$ so that the line $\ell(t_n,i)$ through $t_n$ which is perpendicular to $\ell_i$ already meets $A_i$. Indeed, otherwise we can find a single $j<m_\beta$ and $i\neq i'\in \oo\setm \{i_k:k<n\}$ so that the line $\ell(t_n,i)$ meets $A_i\cap N_{\beta,j}$ already and $\ell(t_n,i')$ meets $A_{i'}\cap N_{\beta,j}$ already. However, this means that $t_n$ is constructible from $\mbb R^2\cap N_{\beta,j}$ so $t_n\in N_{\beta,j}$ as well. This contradicts $t_n\in M_\beta\setm M_{<\beta}$.
 
 So select any $i_n\in \oo\setm \{i_k:k<n\}$ which is not bad for $t_n$ and put $t_n$ into $A_{i_n}$. This finishes the induction and hence the proof of the theorem.
 
\end{proof}

We will proceed now with various new applications of Davies-trees; our aim is to start with simple proofs and proceed to more involved arguments. However, the next four sections of our paper can be read independently.\\

Finally, let us mention explicit applications of Davies-trees from the literature that we are aware of (besides Davies' proof above). 

Arguably, the most important application is S. Jackson and R. D. Mauldin's solution to the \emph{Steinhaus tiling problem} (see \cite{mauldin} or the survey \cite{jackson}). In the late 1950s, H. Steinhaus asked if there is a subset $S$ of $\mbb R^2$ such that every rotation of $S$ tiles the plane or equivalently, $S$ intersects every isometric copy of the lattice $\mbb Z\times \mbb Z$ in exactly one point. Jackson and Mauldin provides an affirmative answer; their ingenious proof elegantly combines deep combinatorial, geometrical and set-theoretical methods (that is, a transfinite induction using Davies-trees). Again, their argument becomes somewhat simpler assuming CH. However, this assumption can be eliminated, as in the Sierpinski-Davies situation, if one uses Davies-trees as a substitute for increasing chains of models. Unfortunately, setting up the application of Davies-trees in this result is not in the scope of our paper so we chose more straightforward applications.

Last but not least, D. Milovich \cite{milovich} polished  Davies, Jackson and Mauldin's technique further to prove a very general result in \cite[Lemma 3.17]{milovich} with applications in set-theoretic topology in mind. In particular, one can guarantee that the Davies-tree $\langle M_\alpha:\alpha<\kappa\rangle$ has the additional property that $\langle N_{\alpha,j}:j<m_\alpha\rangle,\langle M_\beta:\beta<\alpha\rangle\in M_{\alpha}$ for all $\alpha<\kappa$. This extra hypothesis is rather useful in some situations and will come up later in Section \ref{sec:high} as well. Milovich goes on to apply his technique in further papers concerning various order properties \cite{milovichFM1,milovichFM2} and constructions of Boolean algebras \cite{milovichFM3}.


\section{Degrees of disjointness}\label{sec:disjointness}

Let us warm up by proving a simple fact from the theory of almost disjoint set systems. A family of sets $\mc A$ is said to be almost disjoint if $A\cap B$ is finite for all $A\neq B\in \mc A$. There are two well known measures for disjointness:

\begin{dfn} We say that a family of sets $\mc A$ is \emph{$d$-almost disjoint} iff $|A\cap B|<d$ for every $A\neq B\in \mc A$. 

$\mc A$ is \emph{essentially disjoint} iff we can select finite $F_A\subs A$ for each $A\in \mc A$ so that $\{A\setm F_A:A\in \mc A\}$ is pairwise disjoint.
\end{dfn}

There are almost disjoint families which are not essentially disjoint; indeed, any uncountable, almost disjoint family $\mc A$ of infinite subsets of $\oo$ witnesses this. However:

\begin{theorem}[\cite{kopefam}] Every $d$-almost disjoint family $\mc A$ of countable sets is essentially disjoint for every $d\in \mbb N$.
\end{theorem}

The original and still relatively simple argument uses an induction on $|\mc A|$. This induction is eliminated by the use of Davies-trees.

\begin{proof} Let $\kappa=|\mc A|$ and take a Davies-tree $\langle M_\alpha:\alpha<\kappa\rangle$ for $\kappa$ over $\{\mc A,f\}$ where $f:\kappa\to \mc A$ is onto. Note that  $\mc A\subs \bigcup_{\alpha<\kappa} M_\alpha$. Also, recall that $M_{<\alpha}=\bigcup\{N_{\alpha,j}:j<m_{\alpha}\}$ for some $m_\alpha<\oo$ for each $\alpha<\kappa$. 
 
 Our goal is to define a map $F$ on $\mc A$ such that $F(A)\in [A]^{<\omega}$ for each $A\in \mc A$ and $\{A\setm F(A):A\in \mc A\}$ is pairwise disjoint. 

Let $\mc A_\alpha=(\mc A \cap M_\alpha)\setm  M_{<\alpha}$ and $\mc A_{<\alpha}=\mc A \cap M_{<\alpha}$. We define $F$ on each $\mc A_\alpha$ independently so fix $\alpha<\kappa$. 

\begin{obs} $|A\cap (\bigcup \mc A_{<\alpha})|<\omega$ for all $A\in \mc A_\alpha$.
\end{obs}
\begin{proof} Otherwise, there is some $j<m_\alpha$ so that $A\cap \bigcup (\mc A\cap N_{\alpha,j})$ is infinite; in particular, we can select $a\in [A\cap \bigcup (\mc A\cap N_{\alpha,j})]^d$. Note that $\bigcup (\mc A\cap N_{\alpha,j})\subs N_{\alpha,j}$ since each set in $\mc A$ is countable. Hence $a\subs N_{\alpha,j}$ and so $a\in N_{\alpha,j}$ as well. However, $N_{\alpha,j}\models$ ``there is a \emph{unique} element of $\mc A$ containing $a$"  since $\mc A$ is $d$-almost disjoint. So $A\in N_{\alpha,j}\subs \bigcup \mc M_{<\alpha}$ by elementarity which contradicts $A\in \mc A_\alpha$.
\end{proof}

Now list $\mc A_\alpha$ as $\{A_{\alpha,l}:l\in \omega\}$. Let $$F(A_{\alpha,l})=A_{\alpha,l}\cap \bigl(\bigcup \mc A_{<\alpha}\cup \bigcup \{A_{\alpha,k}:k<l\}\bigr)$$ for $l<\omega$. Clearly, $F$ witnesses that $\mc A$ is essentially disjoint.

\end{proof}

In a recent paper, Kojman \cite{kojman} presents a general framework for finding useful filtrations of $\rho$-uniform $\nu$-almost disjoint families where $\rho\geq \beth_\oo(\nu)$. That method is based on Shelah's Revised GCH and an analysis of density functions.
%

\section{Conflict-free colorings}\label{sec:cffree}

The study of colorings of set systems $\mc A$ dates back to the early days of set theory and combinatorics. Let us say that a map $f:\bigcup \mc A\to \lambda$ is a \emph{chromatic coloring} of $\mc A$ iff $f\uhr A$ is not constant for any $A\in \mc A$. \emph{The chromatic number of $\mc A$} is the least $\lambda$ so that there is a chromatic coloring of $\mc A$.  The systematic study of chromatic number problems in infinite setting was initiated by P. Erd\H os and A. Hajnal \cite{EH0}. this section we will only work with families of infinite sets; note that any essentially disjoint family (and so every $d$-almost disjoint as well) has chromatic number 2.

Now, a more restrictive notion of coloring is being \emph{conflict-free}: we say that $f:\bigcup \mc A\to \lambda$ is conflict-free iff $|f^{-1}(\nu)\cap A|=1$ for some $\nu<\lambda$ for any $A\in \mc A$. That is, any $A\in \mc A$ has a color which appears at a unique point of $A$. Clearly, a conflict-free coloring is chromatic. We let $\cfc {(\mc A)}$ denote the least $\lambda$ so that $\mc A$ has a conflict-free coloring. After conflict-free colorings of finite and geometric set systems were studied extensively (see \cite{chei, even, pach, pachh}),  recently A. Hajnal, I. Juh\'asz, L. Soukup and Z. Szentmikl\'ossy \cite{HJSSz} studied systematically the conflict-free chromatic number of infinite set systems. 

We prove the following:

\begin{theorem}\cite[Theorem 5.1]{HJSSz}\label{cfthm} If  $m,d$ are natural numbers and $\mc A\subseteq [\oo_m]^\oo$ is $d$-almost disjoint then
 \begin{eqnarray}
 \cfc {(\mc A)}\le
 {\left\lfloor {\frac{(m+1)(d-1)+1}2}   \right\rfloor+2}.
 \end{eqnarray}
\end{theorem}

Part of the reason to include this proof here is to demonstrate how Davies-trees for $\omega_m$ can provide better understanding of such surprising looking results.

\begin{proof}
 Let $\mc A\subs  [{\omega}_m]^{\omega}$ be $d$-almost disjoint and let $\<M_{\alpha}:{\alpha}<{\omega}_m\>$ be a Davies-tree
for ${\omega}_m$ over $\mc A$. So each initial segment $M_{<\alpha}$ is the finite union $\bigcup_{j<m}N_j$ of elementary submodels. Also, let $K={\left\lfloor {\frac{(m+1)(d-1)+1}2}   \right\rfloor+2}.$

We will define $e_{\alpha}: {\omega}_m\cap (M_{\alpha}\setm M_{<{\alpha}})\to K$ so that 
$e=\bigcup_{{\alpha}<{\omega}_m}e_{\alpha}$ is a conflict-free coloring of  $\mc A$.

 We start by a simple observation:

\begin{obs} If $A\in \mc A\cap (M_{\alpha}\setm M_{<{\alpha}})$ then 
$|A\cap M_{<{\alpha}}|\le m(d-1)$.
\end{obs}
\begin{proof}Indeed, since $\mc A$ is $d$-almost disjoint, any $A\in \mc A$ is uniquely definable from $\mc A$ and any $d$-element subset of $A$. So   $\mc A\in N_j\prec H(\theta)$ and $|A\cap N_j|\ge d$ would imply $A\in N_j$. In turn, $|A\cap N_j|\le d-1$ for all $j<m$.
\end{proof}

Now let $\{A_n:n\in {\omega}\}=\mc A\cap (M_{\alpha}\setm M_{<{\alpha}})$ and pick $$x_n\in A_n\setm \big( \bigcup_{i<n}A_i\cup  \bigcup\{A_k: |A_k\cap \{x_i:i<n\}|\ge d\}\cup M_{<\alpha}\big ).$$   Note that this selection is possible and let $X=\{x_n:n\in {\omega}\}$. Clearly, $1\le |X\cap A_n|\le d+1$ and $x_n\in X\cap A_n\subs \{x_0\dots, x_n\}$.

We say that a color $i<K-1$  is \emph{bad for $A_n$} iff   there is $y\ne z\in A_n\cap (M_{<{\alpha}}\cup
\{x_0,\dots, x_{n-1}\})$ such that  $e(y)=e(z)=i$. In other words, $i$ will not witness that the coloring is conflict free on $A_n$. We say that a color $i<K-1$ is \emph{good for $A_n$} iff  there is a unique $y\in A_n\cap (M_{<{\alpha}}\cup \{x_0,\dots, x_{n-1}\})$ such that  $e(y)=i$. If there happens to be a good color for $A_n$ then we simply let $e(x_n)=\{K-1\}$. This choice ensures that $i$ still witnesses that the coloring $e$ is conflict-free on $A_n$.

Now, suppose that no color $i<K-1$ is good for $A_n$. The observation and the fact that $|X\cap A_n|\le d+1$ implies that there are at most $\left\lfloor \frac{m(d-1)+{d+1}}2\right\rfloor=
\left\lfloor\frac{(m+1)(d-1)+1}2\right\rfloor$ bad colors for $A_n$. Hence, there is at least one $i<K-1$ so that no element $y\in A_n\cap (M_{<{\alpha}}\cup
\{x_0,\dots, x_{n-1}\})$ has color $i.$ We let  $e(x_n)=i$ for such an $i<K-1$; now, $i$ became a good color for $A_n$. This finishes the induction and the proof of the theorem.




\end{proof}
 If GCH holds then the above bound is almost sharp for $d=2$ or for odd values of $d$: $$\cfc {(\mc A)}\ge
 {\left\lfloor {\frac{(m+1)(d-1)+1}2}   \right\rfloor+1}$$ for some $d$-almost disjoint  $\mc A\subseteq [\oo_m]^\oo$, and we recommend the reader to look into \cite{HJSSz} for a great number of open problems. In particular, Theorem \ref{cfthm} gives the bound 4 for $m=d=2$; it is not known if equality could hold, even consistently.


\section{Clouds above the Continuum Hypothesis}\label{sec:clouds}

The next application, similarly to Davies' result, produces a covering of the plane with small sets. However, this argument makes crucial use of the fact that a set of size $\aleph_m$ (for $m\in \mbb N$) can be covered by a Davies-tree such that the initial segments are expressed as the union of $m$ elementary submodels. 

\begin{dfn}We say that $A\subs \mbb R^2$ is a \emph{cloud around a point $a\in \mbb R^2$} iff every line $\ell$ through $a$ intersects $A$ in a finite set.
\end{dfn}

Note that one or two clouds cannot cover the plane; indeed, if $A_i$ is a cloud around $a_i$ for $i<2$ then the line $\ell$ through $a_0$ and $a_1$ intersects $A_0\cup A_1$ in a finite set. How about three or more clouds? The answer comes from a truly surprising result of P. Komj\'ath and J. H. Schmerl:

\begin{theorem}[\cite{clouds} and \cite{schmerl}] The following are equivalent for each $m\in \mbb N$:
\begin{enumerate}
	\item $2^\omega\leq \aleph_m$,
	\item $\mbb R^2$ is covered by at most $m+2$ clouds.
\end{enumerate}
Moreover, $\mbb R^2$ is always covered by countably many clouds.
\end{theorem}

We only prove (i) implies (ii) and follow Komj\'ath's original argument for CH. The fact that countably many clouds always cover $\mbb R^2$ can be proved by a simple modification of the proof below.

\begin{proof} Fix $m\in \omega$ and suppose that the continuum is at most $\aleph_m$. In turn, there is a Davies-tree $\langle M_\alpha:\alpha<\aleph_m\rangle$ for $\mf c$ over $\mbb R^2$ so that $M_{<\alpha}=\bigcup\{N_{\alpha,j}:j<m\}$ for every $\alpha<\aleph_m$. 

Fix $m+2$ points $\{a_k:k<m+2\}$ in $\mbb R^2$ in general position (i.e. no three are collinear). Let $\mc L^k$ denote the set of lines through $a_k$ and let $\mc L=\bigcup\{\mc L^k:k<m+2\}$. We will define clouds $A_k$ around $a_k$ by defining a map $F:\mc L \to [\mbb R^2]^{<\oo}$ such that $F(\ell)\in [\ell]^{<\oo}$ and letting $$A_k=\{a_k\}\cup \bigcup\{F(\ell):\ell\in \mc L^k\}$$ for $k<m+2$. We have to make sure that for every $x\in \mbb R^2$ there is $\ell\in \mc L$ so that $x\in F(\ell)$.

Now, let $\mc L_\alpha=(\mc L \cap M_\alpha)\setm  M_{<\alpha}$  for $\alpha<\aleph_m$. We define $F$ on $\mc L_\alpha$ for each $\alpha<\aleph_m$ independently so fix an $\alpha<\aleph_m$. List $\mc L_\alpha\setm \mc L'$ as $\{\ell_{\alpha,i}:i<\oo\}$ where $\mc L'$ is the set of ${m+2}\choose{2}$ lines determined by $\{a_k:k<m+2\}$. We let $$F(\ell_{\alpha,i})=\bigcup\{\ell\cap \ell_{\alpha,i}:\ell\in \mc L'\cup \{\ell_{\alpha,i'}:i'<i\} \}$$ for $i<\oo$.

We claim that this definition works: fix a point $x\in \mbb R^2$ and we will show that there is $\ell\in \mc L$ with $x\in F(\ell)$. Find the unique $\alpha<\aleph_m$ such that $x\in M_\alpha\setm  M_{<\alpha}$. It is easy to see that $\cup \mc L'$ is covered by our clouds hence we suppose $x\notin \bigcup \mc L'$. Let $\ell_k$ denote the line through $x$ and $a_k$. 

\begin{obs} $| M_{<\alpha}\cap \{\ell_k:k<m+2\}|\leq m$.
\end{obs}
\begin{proof} Suppose that this is not true. Then (by the pigeon hole principle) there is $j<m$ such that $|N_{\alpha,j}\cap \{\ell_k:k<m+2\}|\geq 2$ and in particular the intersection of any two of these lines, the point $x$, is in $N_{\alpha,j}\subs M_{<\alpha}$. This contradicts the minimality of $\alpha$.
\end{proof}

We achieved that $$|\{\ell_k:k<m+2\}\cap (\mc L_\alpha\setm \mc L')|\geq 2$$ i.e. there is $i'<i<\oo$ such that $\ell_{\alpha,i'},\ell_{\alpha,i}\in \{\ell_k:k<m+2\}$. Hence $x\in F(l_{\alpha,j})$ is covered by one of the clouds.

\end{proof}

\section{The chromatic number and connectivity}\label{sec:chrom}

A graph $G$ is simply a set of vertices $V$ and edges $E\subseteq [V]^2$. Recall that the \emph{chromatic number} $\chi(G)$ of a graph $G$ is the least number $\lambda$ such that the vertices of $G$ can be colored by $\lambda$ colors without monochromatic edges. It is one of the fundamental problems of graph theory how the chromatic number affects the subgraph structure of a graph i.e. is it true that large chromatic number implies that certain graphs (like triangles or 4-cycles) must appear as subgraphs? The first result in this area is most likely J. Mycielski's construction of triangle free graphs of arbitrary large finite chromatic number \cite{myc}. 

It was discovered quite early that a lot can be said about uncountably chromatic graphs; this line of research was initiated by P. Erd\H os and A. Hajnal in \cite{EH0}. One of many problems in that paper asked whether uncountable chromatic number implies the existence of \emph{highly connected} uncountably chromatic subgraphs.

 A graph $G$ is called \emph{$n$-connected} iff the removal of less than $n$ vertices leaves $G$ connected. Our aim is to prove P. Komj\'ath's following result:

\begin{theorem}[\cite{kopeconn}] \label{nconn} Every uncountably chromatic graph $G$ contains $n$-connected uncountably chromatic subgraphs for every $n\in \mbb N$.
\end{theorem}

 Fix a graph $\gr$, $n\in\omega$ and consider the set $\mc A$ of all subsets of $V$ inducing maximal $n$-connected subgraphs of $G$. We assume that $G$ has uncountable chromatic number but each $A\in \mc A$ induces a countably chromatic subgraph and reach a contradiction.
 
 First, $\mc A$ essentially covers $G$:
 
 \begin{lemma}
  The graph $G\uhr V\setm \bigcup \mc A$ is countably chromatic.
 \end{lemma}
\begin{proof}
 It is proved in \cite{EH0} that every graph with uncountable chromatic number contains an $n$-connected subgraph; hence the lemma follows.
\end{proof}

We will follow Komj\'ath's original framework however the use of Davies-trees will make our life significantly easier. We let $N_G(v)=\{w\in V: \{v,w\}\in E\}$ for $v\in V$.

\begin{lemma}\label{ordlemma}
 Suppose that $\gr$ is a graph, the sequence $\<A_\xi:\xi<\mu\>$ covers $V$ with countably chromatic subsets so that \[ |N_G(x)\cap  A_{<\xi}|<\oo \textrm{ for all }\xi<\mu \textrm{ and }x\in A_\xi\setm  A_{<\xi}\] where $A_{<\xi}=\bigcup \{A_\zeta:\zeta<\xi\}$. Then $\chi(G)\leq \oo$.

\end{lemma}

\begin{proof} Suppose that $g_\xi:A_\xi\to \oo$ witnesses that the chromatic number of $A_\xi$ is $\leq \oo$. We define $f:V\to \oo\times \oo$ by defining $f\uhp (A_\xi\setm  A_{<\xi} )$ by induction on $\xi<\mu$. If $x\in A_\xi\setm  A_{<\xi}$ then the first coordinate of $f(x)$ is $g_\xi(x)$ while the second coordinate of $f(x)$ avoids all the finitely many second coordinates appearing in $\{f(y):y\in N_G(x)\cap  A_{<\xi}\}$. It is easy to see that $f$ witnesses that $G$ has countable chromatic number.
\end{proof}

Now, our goal is to enumerate $\mc A$ as $\<A_\xi:\xi<\mu\>$ so that the assumptions of Lemma \ref{ordlemma} are satisfied. This will imply that the chromatic number of $G\uhr \bigcup \mc A$ is countable as well which contradicts that $G$ has uncountable chromatic number.

Not so surprisingly, this enumeration will be provided by a Davies-tree but we need a few easy lemmas first.



\begin{obs}\label{inobs}
\begin{enumerate}
\item $A\not\subseteq A'$ for all $A\neq A'\in \mc A$,
\item $|A\cap A'|<n$ for all $A\neq A'\in \mc A$,
\item $|\{A\in \mc A: a\subs A\}|\leq 1$ for all $a\in [V]^{\geq n}$,
\item $|N_G(x)\cap A|<n$ for all $x\in V\setm A$ and $A\in \mc A$.
\end{enumerate}
\end{obs}

The next claim is fairly simple and describes a situation when we can join $n$-connected sets.

\begin{claim}\label{amalg1} Suppose that $A_i\subs V$ spans an $n$-connected subset for each $i<n$ and we can find $Y=\{y_{i,k}:i,k<n\}$ and $X=\{x_k:k<n\}$ distinct points so that $$y_{i,k}\in A_i\cap N_G(x_k)$$ for all $i,k<n$. Then $A=\bigcup \{A_i:i<n\}\cup X$ is $n$-connected.
\end{claim}
\begin{proof}
 Let $F\in [A]^{<n}$ and note that there is a $k<n$ so that $\{y_{i,k},x_k:i<n\}\cap F=\emptyset$. Thus $\bigl (\bigcup \{A_i:i<n\}\cup\{y_{i,k},x_k:i<n\}\bigr )\setm F$ is connected as $A_i\setm F$ is connected for all $i<n$. Finally, if $x_j\in A\setm F$ then $N_G(x_j)\cap \bigcup\{A_i:i<n\}\setm F\neq \emptyset$ so we are done.
\end{proof}

Now, we deduce some useful facts about elementary submodels and maximal $n$-connected sets.

\begin{lemma} Suppose that $N\prec H(\theta)$ with $G\in N$ and $$|N_G(x)\cap N|\geq n$$ for some $x\in V\setm N$. Then $x\in A$ for some $A\in\mc A\cap N$.
\end{lemma}
\begin{proof}
 Let $a\in [N_G(x)\cap N]^{n}$. There is a copy of $K_{n,\omg}$ (complete bipartite graph with classes of size $n$ and $\omg$) which contains $a\cup \{x\}$; to see this, apply Fact \ref{elfact} to $X=\bigcap\{N_G(y):y\in A\}$. As $K_{n,\omg}$ is $n$-connected, there must be $A\in \mc A$ with $a\cup\{x\}\subs A$ as well. Also, there is $A'\in \mc A\cap N$ with $a\subs A'$ by elementarity; as $|A\cap A'|\geq n$ we have $A=A'$ which finishes the proof.
\end{proof}

\begin{lemma}\label{mlemma1} Suppose that $N\prec H(\theta)$ with $G\in N$ and $$|N_G(x)\cap \bigcup(\mc A\cap N)|\geq\omega$$ for some $x\in V\setm N$. Then $x\in A$ for some $A\in\mc A\cap N $.
\end{lemma}
\begin{proof}
 Suppose that the conclusion fails; by the previous lemma, we have $|N_G(x)\cap N|< n$. 
In particular, there is sequence of distinct $A_i\in \mc A\cap N$ for $i<n$ so $$(N_G(x)\cap A_i)\setm N\neq \emptyset$$ for all $i<n$ (as $N_G(x)\cap A$ is finite if $A\in N\cap \mc A$). 

Thus \[ N\models \forall F\in[V]^{<\omega}\exists x\in V\setm F\textrm{ and } y_i\in (A_i\cap N_G(x))\setm F.\]

Now, we can find distinct $\{y_{i,k}:i<n, k<n\}$ and $X= \{x_k:k<n\}$ so that $$y_{i,k}\in A_i\cap N_G(x_k).$$ Finally, $\bigcup \{A_i:i<n\}\cup X$ is $n$-connected by Claim \ref{amalg1} which contradicts the maximality of $A_i$. 
\end{proof}

Finally, lets finish the proof of Theorem \ref{nconn} by defining this ordering of $\mc A$. Take a Davies-tree  $\langle M_\alpha:\alpha<\kappa\rangle$ for $|\mc A|$ over $\{G\}$. In turn, $\mc A\subseteq \bigcup \{M_\alpha:\alpha<\kappa\}$. Recall that  for all $\alpha<\kappa$ there is $(N_{\alpha,j})_{j<m_\alpha}$ so that $$M_{<\alpha}=\bigcup\{N_{\alpha,j}:j<m_\alpha\}$$
with $G\in M_\alpha\cap N_{\alpha,j}$.

Let $\mc A_{<\alpha}=\mc A\cap  M_{<\alpha}$ and $\mc A_\alpha=(\mc A\cap M_\alpha)\setm \mc A_{<\alpha}$ for $\alpha<\kappa$. Well order $\mc A$ as $\{A_\xi:\xi<\mu\}$ so that 
\begin{enumerate}
\item $A_\zeta\in \mc A_{<\alpha},A_\xi\in \mc A\setm \mc A_{<\alpha}$ implies $\zeta<\xi$ and
\item $\mc A_\alpha\setm \mc A_{<\alpha}$ has order type $\leq \omega$
\end{enumerate}

for all $\alpha<\kappa$. 

We claim that the above enumeration of $\mc A$ satisfies Lemma \ref{ordlemma}. By the second property of our enumeration and Observation \ref{inobs} (iv), it suffices to show that $$|N_G(x)\cap\bigcup \mc A_{<\alpha}|<\oo $$ if $x\in A\setm \bigcup \mc A_{<\alpha}$ for all $A\in \mc A_\alpha\setm \mc A_{<\alpha}$ and $\alpha<\kappa$.

However, as $\mc A_{<\alpha}=\bigcup\{\mc A\cap N_{\alpha,j}:j<m_\alpha\}$, this should be clear from applying Lemma \ref{mlemma1} for each of the finitely many models $N_{\alpha,j}$ where $j<m_\alpha$. This finishes the proof of Theorem \ref{nconn}.\\

%

We note that Komj\'ath also proves that every uncountably chromatic subgraph contains an $n$-connected uncountably chromatic subgraph with minimal degree $\omega$; we were not able to deduce this stronger result with our tools.

It is an open problem whether every uncountably chromatic graph $G$ contains a nonempty $\omega$-connected subgraph \cite{koperev} (i.e. removing finitely many vertices leaves the graph connected). These infinitely connected subgraphs might only be countable, as demonstrated by

\begin{theorem}[\cite{trees}] There is a graph of chromatic number $\aleph_1$ and size $\mf c$ such that every uncountable set is separated by a finite set. In particular, every $\omega$-connected subgraph is countable.
\end{theorem}

\section{Davies-trees from countably closed models}\label{sec:high}

In a wide range of problems, we are required to work with countably closed elementary submodels i.e. models which satisfy $[M]^\oo\subseteq M$. Recall that it is possible to find countably closed models $M$ with a given parameter $x\in M$ while the size of $M$ is $\mf c$. A prime example of applying such models is Arhangelskii's theorem \cite{arh_lind}: every first countable, compact space has size at most $\mf c$. In the modern proof of this result \cite{gesel}, a continuous  $\omg$-chain of countably closed models, each of size $\mathfrak c$, is utilized. 




Our main goal in this section is to show that, under certain assumptions, one can construct a sequence of countably closed elementary submodels, each of size $\mf c$, which is reminiscent of Davies-trees while the corresponding models cover structures of size $>\mf c^+$; note that this covering would not be possible by an increasing chain of models of size $\mf c$. 

So, what is it exactly that we aim to show?  First, recall that we have been working with the structure $(H(\theta),\in)$ so far. However, we will now switch to  $(H(\theta),\in,\tri)$ where $\tri$ is some (fixed) well-order on $H(\theta)$, and use its elementary submodels.  We shall see in Section \ref{sec:top} that this can be quite useful e.g. the well order $\tri$ can be used to make \emph{uniform choices in a construction} of say topological spaces and hence the exact same construction can be reproduced by any elementary submodel with the appropriate parameters.

Now, we say that a {\em high Davies-tree for ${\kappa}$ over $x$} 
is a sequence  $\<M_{\alpha}:{\alpha}<{\kappa}\>$ 
of elementary submodels of 
$(H(\theta),\in,\tri)$ for some large enough regular $\theta$ such that 
\begin{enumerate}[(I)]
\item 
$\br M_{\alpha};{\omega};\subs M_{\alpha}$, $|M_{\alpha}|=\mf c$ 
and 
$x\in M_{\alpha}$ for all ${\alpha}<{\kappa}$, \smallskip
\item \label{item:cover}   
$\br {\kappa};{\omega};\subs \bigcup_{{\alpha}<{\kappa}}M_{\alpha}$, and\smallskip
\item for each ${\beta}<{\kappa}$ there are $N_{\beta,j}\prec H(\theta)$ with $[N_{\beta,j}]^\oo\subs N_{\beta,j}$ and 
$x\in N_{\beta,j}$ for $j<\oo$ such that $$\bigcup\{M_{\alpha}:\alpha<\beta\}=\bigcup\{N_{\beta,j}:j<\oo\}.$$
\end{enumerate}

Now, a high Davies-tree is really similar to the Davies-trees we used so far, only that we work with countably closed models of size $\mf c$ (instead of countable ones) and the initial segments $M_{<\beta}$ are countable unions of such models (instead of finite unions). Furthermore, we require that the models cover $\br {\kappa};{\omega};$ instead of $\kappa$ itself. This is because our applications typically require to deal with all countable subsets of a large structure.

One can immediately see that (\ref{item:cover}) implies that $\kappa^\oo=\kappa$ and so high Davies-trees might not exist for some $\kappa$. Nonetheless, a very similar tree-argument to the proof of Theorem \ref{main} shows that high Davies-trees do exist for the finite successors of $\mf c$ i.e. for $\kappa<\mf c^{+\oo}$.  We will not repeat that proof here but present a significantly stronger result.

As mentioned already, some extra set theoretic assumptions will be necessary to prove the existence of high Davies-trees for cardinals above $\mf c^{+\oo}$ so let us recall two notions. We say that $\Box^{}_{\mu}$ holds for a singular $\mu$ iff there is a sequence $\< C_\alpha:\alpha<\mu^+\>$ so that $C_\alpha$ is a closed and unbounded subset of $\alpha$ of size $<\mu$ and $C_\alpha=\alpha\cap C_\beta$ whenever $\alpha$ is an accumulation point of $C_\beta$.  $\Box^{}_{\mu}$ is known as \emph{Jensen's square principle}; R. Jensen proved that  $\Box^{}_{\mu}$ holds for all uncountable $\mu$ in the constructible universe $L$. 

Furthermore, a cardinal $\mu$ is said to be \emph{$\oo$-inaccessible} iff $\nu^\oo<\mu$ for all $\nu<\mu$. Now, our main theorem is the following:

\begin{theorem}\label{tm:niceomegadavies} There is a high Davies-tree $\<M_{\alpha}:{\alpha}<{\kappa}\>$ for ${\kappa}$ over $x$ whenever
\begin{enumerate}
 \item\label{neccond1} ${\kappa}={\kappa}^{\omega}$, and
 \item\label{mupluss1} ${\mu}$ is ${\omega}$-inaccessible, ${{\mu}}^{\omega}={\mu}^+$ and
 $\Box^{}_{\mu}$ holds for all  $\mu$ with $\mf c< {\mu}<{\kappa}$ and  $\cf({\mu})={\omega}$.
 \end{enumerate}\smallskip
 Moreover, the high Davies-tree $\<M_{\alpha}:{\alpha}<{\kappa}\>$ can be constructed so that 
 \begin{enumerate}
 \setcounter{enumi}{2}
  \item\label{pp3} $\<M_{\alpha}:{\alpha}<{\beta}\>\in M_\beta$ for all $\beta<\kappa$, and 
  \item\label{pp4} $\bigcup\{M_\alpha:\alpha<\kappa\}$ is also a countably closed elementary submodel of $H(\theta)$.
 \end{enumerate}
 
 \end{theorem}

We will say that $\<M_{\alpha}:{\alpha}<{\kappa}\>$ is a \emph{sage Davies-tree} if it is a high Davies-tree satisfying the extra properties \ref{pp3}. and \ref{pp4}. above. 
Finally, let us remark that if one only aims to construct high Davies-trees (which are not necessary sage) then slightly weaker assumptions than \ref{neccond1}. and \ref{mupluss1}. suffice; see Theorem \ref{thm:onlyhigh} for further details. 

In order to state a rough corollary, recall that \ref{neccond1}. and \ref{mupluss1}. are satisfied by all $\kappa$ with uncountable cofinality in the constructible universe. Hence:

\begin{corollary}
 If $V=L$ then there is a sage Davies-tree for ${\kappa}$ over $x$ for any cardinal $\kappa$ with uncountable cofinality. 
\end{corollary}

\medskip

Our plan is to postpone the proof of Theorem \ref{tm:niceomegadavies} to the appendix in Section \ref{sec:appendix} because it involves much more work than proving the existence of usual Davies-trees. Indeed, we need a completely different approach than the tree-argument from the proof of Theorem \ref{main}.  Furthermore, we believe that the proof itself gives no extra insight to the use of high Davies-trees in practice. 

So instead, we start with applications first in the next couple of sections. We hope to demonstrate that the existence of high/sage Davies-trees can serve as a simple substitute for technically demanding applications of $\Box_\mu$ and cardinal arithmetic assumptions. 

Some of the presented applications will show that condition (\ref{mupluss1}) in Theorem \ref{tm:niceomegadavies} cannot be weakened to say the Generalized Continuum Hypothesis i.e. $2^\lambda=\lambda^+$ for all infinite cardinal $\lambda$.  We will show that \emph{consistently, GCH holds and there are no high Davies-trees for any $\kappa$ above $\aleph_\oo$}; we will use a supercompact cardinal for this consistency result (see Corollary \ref{cor:nohigh}).

\medskip

Finally, let us point the interested reader to a few related construction schemes which were used to produce similar results: the technique of Jensen-matrices \cite{foreman}, simplified morasses \cite{velleman} and cofinal Kurepa-families \cite[Definition 7.6.11]{walks}. In particular, the latter method was used to produce Bernstein-decompositions of topological spaces (see Section \ref{sec:B}) and splendid spaces (see Section \ref{sec:top}). 

\section{More on large chromatic number and the subgraph structure}\label{sec:morechrom}

We start with an easy result about large chromatic number and subgraphs, very much in the flavour of Section \ref{sec:chrom}. This result will also demonstrate that high Davies-trees might not exist for any $\kappa\ge\aleph_{\oo+1}$ even if GCH holds.

Let $[\aleph_0,\mf c^+]$ denote the graph on vertex set $U\dot\cup V$ where $|U|=\aleph_0$ and $|V|=\mf c^+$, and edges $\{uv:u\in U,v\in V\}$.

\begin{theorem}\label{highchrom}
 Suppose that $G$ is a graph of size $\lambda$ without copies of $[\aleph_0,\mf c^+]$. If there is a high Davies-tree for any $\kappa\ge \lambda$ over $G$ then $\chi(G)\leq \mf c$.
\end{theorem}

The above theorem is a rather weak version of \cite[Theorem 3.1]{colnum}; there GCH and a version of the square principle was assumed to deduce a more general and stronger result.

\begin{proof}
Let $\langle M_\alpha:\alpha<\kappa\rangle$ be a high Davies-tree for $\kappa\ge \lambda$ over $G$. Without loss of generality, we can suppose that the vertex set of $G$ is $\lambda$ and utilize $\lambda\subseteq \bigcup\{M_\alpha:\alpha<\kappa\}$.

Our plan is to define  $f_\alpha: \lambda\cap M_\alpha\setm M_{<\alpha}\to \mf c$ inductively so that $f=\bigcup\{f_\alpha:\alpha<\kappa\}$ witnesses $\chi(G)\leq \mf c$. That is, $f$ is a chromatic coloring: $f(u)\neq f(v)$ whenever $uv$ is an edge of $G$.
 
 Suppose we defined colorings $f_\alpha$ for $\alpha<\beta$ such that $f_{<\beta}=\bigcup\{f_\alpha:\alpha<\beta\}$ is chromatic. List $\lambda\cap M_\beta\setm M_{<\beta}$ as $\{v_\xi:\xi<\nu\}$ for some $\nu\le \mf c$ and let's define $f_\beta(v_\xi)$ by induction on $\xi<\nu$. Our goal is to  choose $f_\beta(v_\xi)$ from $\mf c\setm \{f_\beta(v_\zeta):\zeta<\xi\}$ to make sure that $f_\beta$ is chromatic, and also that $f_\beta(v_\xi)\neq f_{<\beta}(u)$ for any $u\in \lambda\cap M_{<\beta}$ such that $uv_\xi$ is an edge. If we can do this then $f_{<\beta}\cup f_\beta$ is chromatic as well.
 
Our first requirement is easy to meet since we only want to avoid $\{f_\beta(v_\zeta):\zeta<\xi\}$, a set of size $<\mf c$. The next claim implies that there are not too many edges from $v_\xi$ into $\lambda\cap M_{<\beta}$:
 
 \begin{claim}
  If $v\in \lambda\cap M_{\beta}\setm M_{<\beta}$ then $\{u\in \lambda\cap M_{<\beta}: uv$ is an edge$\}$ is countable.
 \end{claim}
\begin{proof} Recall that $M_{<\beta}=\bigcup\{N_{\beta,j}:j<\oo\}$ so that $G\in N_{\beta,j}$ and $[N_{\beta,j}]^\oo\subs N_{\beta,j}$. Let $N(v)=\{u\in \lambda: uv$ is an edge$\}$. If $N(v)\cap M_{<\beta}$ is uncountable then  we can find $A\in[N(v)\cap N_{\beta,j}]^\oo$ for some $j<\oo$. Since $A\in N_{\beta,j}$, the set $B=\bigcap \{N(u):u\in A\}$ is also an element of $N_{\beta,j}$ and $v\in B$. We claim that $|B|>\mf c$; indeed, otherwise $B\subseteq N_{\beta,j}$ and so $v\in N_{\beta,j}$ which contradicts the choice of $v$. However, any point in $A$ is connected to any point in $B$ which gives a copy of the complete bipartite graph  $[\aleph_0,\mf c^+]$ in $G$. This again is a contradiction.
\end{proof}

 So, we can choose $f_\beta(v_\xi)$ and extend our chromatic coloring as desired. In turn, this finishes the induction and the proof of the theorem.

\end{proof}

%

While we do not claim that the existence of a high Davies-tree for $\kappa\ge \lambda$ implies that there are high Davies-trees for $\lambda$ too, the above proof demonstrates that practically we can work with a high Davies-tree for $\kappa$ as a high Davies-tree for $\lambda$.

\begin{corollary}\label{cor:nohigh}
 Consistently, relative to a supercompact cardinal, GCH holds and there are no high Davies-trees for any $\kappa\ge \aleph_{\oo+1}$.
\end{corollary}
\begin{proof}
 This will follow from \cite[Theorem 4.7]{splitting}: Consistently, relative to a supercompact cardinal, GCH holds and there is a graph $G$ on vertex set $\aleph_{\oo+1}$ of chromatic number $\mf c^+$ so that $G$ contains no copies of $[\aleph_0,\mf c^+]$. 
 
 Now, working in the above model, fix $\kappa\ge \aleph_{\oo+1}$, a cardinal $\theta$ and a well order $\tri$ of $H(\theta)$. Let $G^*$ be the unique $\tri$-minimal graph on vertex set $\aleph_{\oo+1}$ of chromatic number $\mf c^+$ so that $G^*$ contains no copies of $[\aleph_0,\mf c^+]$. If $\langle M_\alpha:\alpha<\kappa\rangle$ is a high Davies-tree for $\kappa\ge \aleph_{\oo+1}$ from $(H(\theta),\in,\tri)$ then $\langle M_\alpha:\alpha<\kappa\rangle$ must be a high Davies-tree over $G^*$; indeed, $G^*$ is uniquely definable from $\aleph_{\oo+1}$ and $\mf c^+$ using $\tri$, and these parameters are in each relevant submodel of $(H(\theta),\in,\tri)$.
 
 Hence, there are no high Davies-tree for any $\kappa\ge \aleph_{\oo+1}$ in that model by our Theorem \ref{highchrom}.
\end{proof}

\section{Coloring topological spaces}\label{sec:B}

Our first application concerns a truly classical result due to F. Bernstein from 1908 \cite{Bernstein}: there is a coloring of $\mbb R$ with two colors such that no uncountable Borel set is monochromatic. In other words, the family of Borel sets in $\mbb R$ has chromatic number 2. Indeed, list all the uncountable Borel sets as $\{B_\alpha:\alpha<\mf c\}$ and inductively pick distinct $x_\beta,y_\beta\in B_\beta\setm \{x_\alpha,y_\alpha:\alpha<\beta\}$. This can be done since each $B_{\beta}$ contains a Cantor subspace and so has size continuum. Now any map $f:\mbb R\to 2$ that sends $\{x_\alpha:\alpha<\mf c\}$ to $0$ and $\{y_\alpha:\alpha<\mf c\}$ to $1$ is as desired.

 Now, let  $\mc C(X)$ denote the set of Cantor subspaces of an arbitrary topological space $(X,\tau)$. Can we extend Bernstein's theorem to general topological spaces? The above simple argument certainly fails if there are more than $\mf c$ many Cantor subspaces.

\begin{theorem}Suppose that $(X,\tau)$ is a Hausdorff topological space of size $\kappa$. If there is a high Davies-tree for $\kappa$ over $(X,\tau)$ then there is a coloring $f:X\to \mf c$   so that $f[C]=\mf c$ for any $C\in \mc C(X)$.
\end{theorem}

Let us call such a function $f:X\to \mf c$ \emph{a Bernstein-decomposition of $X$}. Now, if $|X|< \mf c^{+\oo}$ or, more generally, $\kappa$ satisfies the assumptions of Theorem \ref{tm:niceomegadavies} then Bernstein-decompositions for $X$ exist. 

Originally, the consistency of ``any Hausdorff space $X$ has a Bernstein-decomposition'' was originally proved by W. Weiss \cite{weiss} (see \cite{weisssurv} for a survey). For an alternative proof using cofinal Kurepa-families, see \cite[Theorem 7.6.31]{walks}. Let us also mention that a more general (and more technical) coloring result was achieved in \cite[Theorem 3.5]{splitting} using similar but stronger assumptions to the ones in our  Theorem \ref{tm:niceomegadavies}; it is likely that one can do the same using high Davies-trees. Finally, S. Shelah \cite{shelah} showed (using a supercompact cardinal) that consistently there is a 0-dim, Hausdorff space $X$ of size $\aleph_{\oo+1}$ without a Bernstein-decomposition.


\begin{proof} Let  $\langle M_\alpha:\alpha<\kappa\rangle$ be a high Davies-tree for $\kappa$ over $X$. In turn, $X$ and $[X]^\oo$ are covered by $\bigcup \{M_\alpha:\alpha<\kappa\}$. We let $\mc C_{\alpha}=\mc C(X)\cap M_\alpha\setm M_{<\alpha}$, $X_\alpha=X\cap  M_\alpha\setm M_{<\alpha}$ and $X_{<\alpha}=X\cap M_{<\alpha}$. 

\begin{claim}\label{cantorclm} Suppose that $C\in \mc C(X)$ and $C\cap X_{<\alpha}$ is uncountable. Then there is a $D\in M_{<\alpha}\cap \mc C(X)$ such that  $D\subseteq C$.
\end{claim}
\begin{proof} Indeed, $M_{<\alpha}=\bigcup\{N_{\alpha,j}:j<\oo\}$ and each $N_{\alpha,j}$ is $\oo$-closed. So there must be an $j<\oo$ such that $C\cap N_{\alpha,j}$ is uncountable. Find $A\subseteq C\cap N_{\alpha,j}$ which is countable and dense in $C\cap N_{\alpha,j}$. Note that $A$ must be an element of $N_{\alpha,j}$ as well and hence, the uncountable closure $\bar A$ of $A$ is an element of $N_{\alpha,j}$ (since $\tau\in N_{\alpha,j}$). Now, we can pick $D\subseteq \bar A\subseteq C$ such that $D\in N_{\alpha,j}\cap \mc C(X)\subseteq M_{<\alpha}\cap \mc C(X)$.
\end{proof} 
 
 We define $f_{\alpha}:X_{\alpha}\to \mf c$ so that $f_{\alpha}[C]=\mf c$ for any $C\in \mc C_{\alpha}$ so that $C\cap X_{<\alpha}$ is countable. This can be done just like Bernstein's original theorem; indeed let $$\mc C^*_\alpha=\{C\cap X_\alpha: C\in \mc C_{\alpha}, |C\cap X_{<\alpha}|\leq \oo\}.$$ $\bigcup \mc C^*_\alpha\subseteq X\cap M_\alpha$ and if $\mc C^*_\alpha\neq \emptyset$ then $|\bigcup \mc C^*_\alpha|=\mf c$. Moreover,   $|\mc C^*_\alpha|\leq \mf c$ and $\mc C^*_\alpha$ is $\mf c$-uniform i.e. each element has size $\mf c$.  So, we can use the same induction as Bernstein to find $f_\alpha$.
 

We claim that $f=\bigcup\{f_\alpha:\alpha<\kappa\}$ satisfies the requirements. Indeed, suppose that $C\in \mc C(X)$ and let $\alpha$ be minimal so that $D\subseteq C$ for some $D\in \mc C_\alpha$.  Claim \ref{cantorclm} implies that  $C\cap X_{<\alpha}$ is countable and hence $\mf c=f_\alpha[D]\subseteq f[C]$.

\end{proof}

We remark that in the examples preceding this section, Davies-trees were mainly used to find well behaving enumerations of almost disjoint set systems. This is certainly not the case here for the family $\mc C(X)$. 


\section{Saturated families}\label{sec:saturated}

The following still open problem stands out in the theory of almost disjoint sets: is there, in ZFC, an infinite almost disjoint family $\mc A\subseteq [\oo]^\oo$ so that any $B\in [\oo]^\oo$ either contains an element from $\mc A$ or is covered mod finite by a finite subfamily of $\mc A$.  Such families were introduced by Erd\H os and S. Shelah \cite{satmad} and are called saturated or completely separable. 

Now, in more generality:

\begin{definition}
Let ${\kappa}$ be a cardinal and $\mc F\subs \br {\kappa};{\omega};$.
We say that a family $\mc A$ is {\em $\mc F$-saturated} if
$\mc A\subs \mc F$  and for all $F\in \mc F$ either
\begin{itemize}
 \item $A\subs F$ for some $A\in\mc A $, or
 \item $F\subs^* \bigcup \mc A'$ for some $\mc A'\in \br  \mc A;<{\omega};$.
\end{itemize}

%
%
\end{definition}

So the completely separable families mentioned at the beginning of the section are exactly the almost disjoint $\br {\oo};{\omega};$-saturated families. Our goal is to prove

\begin{theorem}\label{tm:saturated}
If CH holds and there is a high Davies-tree 
for ${\kappa}$ then there is an almost disjoint $\br {\kappa};{\omega};$-saturated family.
\end{theorem}

First, our theorem gives Baumgartner's result that CH implies the existence of  almost disjoint $\br {\aleph_n};{\oo};$-saturated families for all finite $n$ (see the remark at Problem 37 in \cite{EH2}). Second, in \cite{gold}, similar assumptions (i.e. $\Box_\mu$ and a weak form of $\mu^\oo=\mu^+$ for $\cf(\mu)<\mu$) were used to deduce the consistency of ``there is an almost disjoint $\br {\kappa};{\omega};$-saturated family for all $\kappa$''.

\begin{proof}
Let $\<M_{\alpha}:{\alpha}<{\kappa}\>$ be a  high Davies-tree 
for ${\kappa}$.  So,  
$
M_{<{\alpha}}=\bigcup_{i<{\omega}}N_{\alpha,i}
$
for some   $N_{\alpha,i}\prec  H(\theta)$
with    $\br N_{\alpha,i};{\omega};\subs N^{\alpha}_i,$.

By transfinite recursion on ${\alpha}\le {\kappa}$ we will define 
families  $\mc A_{\alpha}$ such that 
\begin{enumerate}[(1)]
\item $\mc A_{\alpha}\subseteq M_{<\alpha}\cap [\kappa]^\oo$ is almost disjoint,
\item $\mc A_{\alpha}\subs \mc A_{\beta}$ if ${\alpha}<{\beta}$, and
\item $\mc A_\alpha$ is $\mc F_{\alpha}$-saturated where $\mc F_{\alpha}=M_{<\alpha}\cap \br {\kappa};{\omega};$.
\end{enumerate} If we succeed then $\mc A_\kappa$ is the desired almost disjoint $\br {\kappa};{\omega};$-saturated family since $\mc F_\kappa=[\kappa]^\oo$.

In limit steps we can simply take unions so suppose that $\mc A_{\alpha}$ is defined and we will find $\mc A_{\alpha+1}$.

Using CH, we can enumerate $(M_{\alpha}\setm M_{<{\alpha}})\cap \br {\kappa};{\omega};$ as  $\{H_{\xi}:{\xi}<{\omega}_1\}$. By induction on ${\xi}\le{\omega}_1$, we define  families $\mc B_{\xi}$ such that such that 
\begin{enumerate}[(i)]
\item $\mc B_0=\mc A_{\alpha}$, $|\mc B_\xi\setm \mc A_\alpha|\leq \oo$ and  $\mc B_\eta\subs \mc B_{\xi}$ if $\eta<{\xi}<\omg$,
\item $\mc B_{\xi}\subseteq M_{<\alpha+1}\cap [\kappa]^\oo$ is almost disjoint, and
\item $\mc B_{\xi}$ is $\big (\mc F_{\alpha}\cup\{H_{\zeta}:{\zeta}<{\xi}\}\big)$-saturated.
\end{enumerate}

Note that if we succeed then $\mc A_{\alpha+1}=\mc B_\omg$ is the desired family. As before, in limit steps we just take  unions so assume that $\mc B_{\xi}$ is defined and consider $H_{\xi}$. We distinguish two cases:

\medskip

\noindent {\bf Case 1. }{ \em  There a finite subset $\mc C\subs \{N_{\alpha,j}\cap {\kappa}:j<{\omega}\}
\cup (\mc B_{\xi}\setm \mc A_{\alpha} )$ such that $H_{\xi}\subs^*\bigcup \mc C$.}

\medskip
We will show that $\mc B_{{\xi}+1}=\mc B_{\xi}$ satisfies the requirements i.e. $\mc B_{\xi}$ is $\big (\mc F_{\alpha}\cup\{H_{\zeta}:{\zeta}\leq {\xi}\}\big)$-saturated. Of course we only need to deal with $H_\xi$.

Note that whenever $N_{\alpha,j}\cap {\kappa}\in \mc C$ then $D_j=H_{\xi}\cap N_{\alpha,j}\in N_{\alpha,j}$ since 
$\br N_{\alpha,j};{\omega};\subs N_{\alpha,j}$, and so $D_j\in \mc F_{\alpha}$. 
Since $\mc A_{\alpha}$ is $\mc F_{\alpha}$-saturated 
either (a) $A\subs D_j$ for some $A\in \mc A_\alpha$, or (b) 
 $D_j\subs^* \bigcup \mc A_j'$ for some $\mc A_j'\in \br  \mc A_{\alpha};<{\omega};$.

In case (a) holds for any $j$, then $\mc B_{{\xi}+1}=\mc B_{\xi}$ clearly satisfies the requirements. So we can assume that 
\begin{displaymath}
 H_{\xi}\cap N_{\alpha,j}\cap {\kappa}\subs^*  \bigcup \mc A_j'
\end{displaymath}
for all $N_{\alpha,j}\cap {\kappa}\in \mc C$.

Let $\mc B'=\bigcup\{\mc A_j':N_{\alpha,j}\cap {\kappa}\in \mc C \}\cup 
\big(\mc C\cap (\mc B_{\alpha}\setm \mc A_{\alpha})\big).$
Then $\mc B'\in [\mc B_\xi]^{<\oo}$ and
\begin{eqnarray}
 H_{\xi}\subs^* \bigcup \mc C\subs^*\bigcup \mc B', 
\end{eqnarray}
so $\mc B_{{\xi}+1}=\mc B_{\xi}$ satisfies the requirements.

\medskip

\noindent {\bf Case 2.}  {\em $H_{\xi}\setm \bigcup \mc C$ is infinite for all finite 
subset $\mc C\subs \{N_{\alpha,j}\cap {\kappa}:j<{\omega}\}
\cup (\mc B_{\xi}\setm \mc A_{\alpha} )$.}

\medskip
Now, there is an infinite $B_{\xi}\subs H_{\xi}$ such that 
$B_{\xi}\cap \bigcup \mc C$ is finite for all finite 
subset $\mc C\subs \{N_{\alpha,j}\cap {\kappa}:j<{\omega}\}
\cup (\mc B_{\xi}\setm \mc A_{\alpha} )$. Indeed, we can list $\{N_{\alpha,j}\cap {\kappa}:j<{\omega}\}
\cup (\mc B_{\xi}\setm \mc A_{\alpha} )$ as $\{C_n:n\in \oo\}$, pick $b_n\in H_\xi\setm \bigcup\{C_k:k<n\}\cup\{b_k:k<n\}$ and set $B_\xi=\{b_n:n\in \oo\}$.

First, $B_\xi\in M_\alpha$ since $H_\xi\in M_\alpha$ and $M_\alpha$ is $\oo$-closed. So if we let  $\mc B_{\xi+1}=\mc B_{\xi}\cup \{B_{\xi}\}$ then $\mc B_{\xi+1}\subseteq  M_{<\alpha+1}\cap [\kappa]^\oo$. Furthermore, $\mc B_{\xi+1}$ is clearly $\big (\mc F_{\alpha}\cup\{H_{\zeta}:{\zeta}\le{\xi}\}\big)$-saturated.

Finally, we need that $\mc B_{\xi+1}$ is almost disjoint; $B_\xi\cap B$ is clearly finite for all $B\in \mc B_\xi\setm \mc A_\alpha$. If $B\in \mc A_\alpha$ then $B\in M_{<\alpha}$ so $B\in N_{\alpha,j}$ for some $j<\oo$. So $B_\xi\cap B$ is finite again. 
\medskip

This ends the construction of $\<B_{\xi}:{\xi}\le {\omega}_1\>$ and, as mentioned before, we let $\mc A_{{\alpha}+1}=\mc B_{{\omega}_1}$. This finishes the main induction and the proof of the theorem.

\end{proof}

It is still unknown, if one can prove that
there is an almost disjoint $\br {\kappa};{\omega};$-saturated family for all $\kappa$ purely in ZFC; although, significant evidence   hints that the answer is yes \cite{sh935}.

\section{The weak Freese-Nation property}\label{sec:freese}

The next theorem we present concerns the structure of the poset $([\kappa]^\oo,\subseteq)$:

\begin{theorem}\label{FN}
Suppose that CH holds and there is a high Davies-tree for $\kappa$. Then there is a function $F$ with domain $[\kappa]^\oo$ so that \begin{enumerate}
\item $|F(a)|\leq \oo$, and
\item if $a\subseteq b\in [\kappa]^\oo$ then there is $c\in F(a)\cap F(b)$ with $a\subseteq c\subseteq b$.
 \end{enumerate}
 \end{theorem}
 
  The existence of a map $F$ as above is usually stated by saying that the partial order $([\kappa]^\oo,\subseteq)$ has the \emph{weak Freese-Nation-property}. We refer the interested reader to \cite{foreman,freese} and also \cite{milovichFM1,milovichFM2, milovichFM3} for more on the Freese-Nation-property. In \cite{freese}, matrices of elementary submodels are used (very much in the flavour of Section \ref{sec:appendix}), and in the latter two papers, Milovich's version of Davies-trees (i.e. long $\lambda$-approximation sequences) make an appearance.

 
 \begin{proof}
 Let  $\langle M_\alpha:\alpha<\kappa\rangle$ be a high Davies-tree for $\kappa$. We inductively define $F$ on $[\kappa]^\oo\cap M_{<\alpha}$ so that $a\in F(a)$, and (1) and (2) are satisfied when restricted to $a\subseteq b\in [\kappa]^\oo\cap M_{<\alpha}$.
 
 At limit steps, we simply take unions so suppose that $F$ is given on $M_{<\alpha}$ and we define $F(a)$ for $a\in [\kappa]^\oo\cap M_{\alpha}\setm M_{<\alpha}$. To this end, list $[\kappa]^\oo\cap M_{\alpha}\setm M_{<\alpha}$ as $\{a_\xi:\xi<\omg\}$. Now, let $$F(a_\xi)=\{a_\zeta:\zeta\leq \xi\}\cup \bigcup\{F(a_\xi\cap N_{\alpha,j}):j<\oo\}$$ for $\xi<\omg$.
 
 Now, let us check (2) since $F(a_\xi)$ is clearly countable. Let $b\in [\kappa]^\oo\cap M_{<\alpha+1}$; if $b\in M_\alpha\setm M_{<\alpha}$ then $b\in F(b)\cap F(a_\xi)$ and we are done. So suppose that $b\in M_{<\alpha}$ and in turn $b\in N_{\alpha,j}$ and $b\subseteq N_{\alpha,j}$ for some $j<\oo$. If $b\subseteq a_\xi$ then $b\subseteq a_\xi\cap N_{\alpha,j}\in M_{<\alpha}$ so $$b\cap N_{\alpha,j}\subseteq c\subseteq  a_\xi\cap N_{\alpha,j}$$ for some $c\in F(b)\cap F(a_\xi\cap N_{\alpha,j})$.  Hence $b\subseteq c\subseteq a_\xi$ and $c\in F(b)\cap F(a)$. 
 
 If $a_\xi \subseteq b$ then $a_\xi\in N_{\alpha,j}$ as well since $N_{\alpha,j}$ is $\oo$-closed. However, this contradicts $a_\xi\in M_\alpha\setm M_{<\alpha}$.
\end{proof}

We remark that  $([\kappa]^\oo,\subseteq)$ may fail the weak Freese-Nation property, in particular if GCH and $(\aleph_{\oo+1},\aleph_\oo)\onto (\aleph_1,\aleph_0)$, a particular instance of Chang's Conjecture, holds (see \cite[Theorem 12]{freese}).

\section{Locally countable, countably compact spaces} \label{sec:top}

In our final application, we show how to use sage  Davies-trees to construct nice topological spaces. 
Let us recall some topological properties first. It is well known that the set of countable ordinals $\omg$ with the topology inherited from their usual order satisfies the following properties: 
\begin{enumerate}[(i)]
 \item 0-dimensional and $T_2$ i.e. there is a basis of closed and open sets, and any two points can be separated by disjoint open sets;
 \item locally countable and locally compact i.e. any point has a countable and compact neighbourhood;
 \item\label{c:obound} countable sets have compact closure.
\end{enumerate}
  
 Topological spaces with the above three properties are often called \emph{splendid} in the literature \cite{juhawe,moreon, good}. The study of splendid spaces dates back to a long standing open problem of E. van Douwen \cite{douwen}: what are the possible sizes of $T_3$, locally countable and countably compact spaces? Each splendid space satisfies van Douwen's requirements since (\ref{c:obound}) implies that any countable set has an accumulation point i.e. the space is countably compact.
 
 I. Juh\'asz, Zs. Nagy and W. Weiss \cite{juhawe} built splendid spaces of cardinality $\kappa$ for   any $\kappa$ with uncountable cofinality using $V=L$; P. Nyikos later observed that the proof only uses instances of $\Box_\mu$ and certain cardinal arithmetic assumptions \cite{moreon}. An alternative construction is outlined in \cite{dowsurv} using cofinal Kurepa-families. Finally, Juh\'asz, S. Shelah and L. Soukup \cite{moreon} showed that it is consistent that there are no splendid spaces of size bigger than $\aleph_\oo$. Nonetheless, van Douwen's problem is still open in ZFC. 
 
 \medskip
 
 We will now present a straightforward and self contained construction of splendid spaces using sage  Davies-trees:

\begin{theorem}\label{thm:splendid} Suppose CH holds and there is a sage Davies-tree for $\kappa$. Then there is a splendid space of size $\kappa$.

\end{theorem}

The most straightforward approach would be to define the required topology on larger and larger portions of $\kappa$ by inductively making the closure of each countable set compact. That is, at each step we consider a countable set $a$, and if the closure of $a$ is not compact yet then we add an extra point to compactify this closure. 

The first step is to show how to make a small (size $\mf c$) space splendid by adding a small number of new points; this is essentially \cite[Lemma 7]{juhawe} but we include a proof here as well:

\begin{lemma}\label{lm:ext}Assume that CH holds. Suppose that $\mc Y_0=(Y_0,\rho_0)$ is a 0-dimensional, $T_2$, locally compact and locally countable  space of size $\omg$ which is also $\oo$-fair i.e. countable sets have countable closure. Then there is $Y\supseteq Y_0$ of size $\omg$ and a topology $\rho\supseteq \rho_0$ on $Y$ so that $\mc Y=(\mc Y,\rho)$ is splendid and $\rho_0=\rho\cap \mc P(Y_0)$.
\end{lemma}
\begin{proof}Let $Y=Y_0\cup\{y_\xi:\xi<\omg\}$ and enumerate $[Y]^\oo$ as $\{a_\xi:\xi<\omg\}$ so that $y_\zeta\in a_\xi$ implies $\zeta<\xi$. 

We will use $y_\xi$ to make the closure of $a_\xi$ compact. More precisely, we define 
topological spaces $\mc Y_{\xi}=\<Y_{\xi},\rho_{\xi}\>$ for $\xi\leq \omg$, where 
$Y_{\xi}=Y_0\cup\{y_{\zeta}:{\zeta}<{\xi}\}$,
such that 
\begin{enumerate}[(i)]
 \item\label{enum:z} $\mc Y_{\xi}$ is a 0-dimensional, $T_2$, locally compact and locally countable  space,
\item\label{enum:zxizieta}  if ${\zeta}<\xi$ then $\rho_{\zeta}=\rho_{\xi}\cap \mc P(Y_{\zeta})$, and
\item\label{conv} $\cl_{\rho_{\xi+1}}(a_\xi)$ is compact.
\end{enumerate}
In limit steps we simply take unions, so we need to construct $\mc Y_{\xi+1}$ assuming that $\mc Y_{\xi}$ is defined. Consider the countable, closed $A=\cl_{\rho_\xi}(a_\xi)$ and suppose that $A$ is not compact; otherwise, we let $\rho_{\xi+1}$ be the topology generated by $\rho_\xi\cup\{\{y_\xi\}\}$.

Now, note that $A$ consists of $\cl_{\mc Y_0}(a_\xi)$ and some points of the form $y_\zeta$ for $\zeta<\xi$.  It is an easy exercise to prove that there is a countable, $\rho_\xi$-clopen $W\subseteq Y_\xi$ which contains $A$; indeed, first find a countable open $V\subseteq Y_\xi$ containing $A$. Then working in the countable set $\cl_{\rho_\xi}(V)$, find a clopen $A\subseteq W$ as desired by inductively covering more and more points of $A$ while avoiding more and more points from $\cl_{\rho_\xi}(V)\setm A$. 

Now, we define $\rho_{\xi+1}$ to be the topology generated by 
\begin{displaymath}
\rho_\xi\cup\{\{y_\xi\}\cup W\setm F:F\textmd{ is compact in }\rho_\xi\}. 
\end{displaymath}

So $ \{y_\xi\}\cup W$ with the subspace topology (from $\rho_{\xi+1}$) is really just the one-point compactification of $W$ (with $\rho_\xi$). The fact that $W$ is countable and clopen implies that $\rho_{\xi+1}$ is locally countable and 0-dimensional. Furthermore, the set $W\cup \{y_\xi\}$ is clearly compact and clopen so $\rho_{\xi+1}$ is locally compact and $\cl_{\rho_{\xi+1}}(a_\xi)$ must be compact as well.

Now, $\mc Y=\mc Y_\omg$ is as desired.
 
\end{proof}

Now, using the above lemma we can prove the theorem:

\begin{proof}[Proof of Theorem \ref{thm:splendid}]
Let $\<M_{\alpha}:{\alpha}<{\kappa}\>$ be a sage  Davies-tree for ${\kappa}$ from $(H(\theta),\in,\tri)$ where $\tri$ is a well order of $H(\theta)$. As always, let
$
M_{<{\alpha}}=\bigcup_{j<{\omega}}N_{\alpha,j}
$
for some   $N_{\alpha,j}\prec (H(\theta),\in,\tri)$
with    $\br N_{\alpha,j};{\omega};\subs N_{\alpha,j}$. 
%
%
%

By transfinite recursion on ${\alpha}\le {\kappa}$ we will define topological spaces
$\mc X_{\alpha}=\<X_{\alpha}, \tau_{\alpha}\>$,
where $X_{\alpha}= {\kappa}\cap M_{<{\alpha}}$, such that
\begin{enumerate}[(1)]
 \item\label{enum:topprop} $\mc X_{\alpha}$ is a 0-dimensional, $T_2$, locally compact and locally countable space,
 \item\label{enum:fair} $\mc X_{\alpha}$ is $\oo$-fair,
 \item\label{enum:reconstr} $\tau_\alpha$ is \emph{reconstructible} i.e. $\tau_\alpha$ can be uniquely defined from $\kappa$ and $\<M_\zeta:\zeta<\alpha\>$, 
 

 \item\label{backint} if ${\alpha}<{\beta}$ then $\tau_{\alpha}=\tau_{\beta}\cap \mc P(X_{\alpha})$, and
\item\label{enum:Aclosed} if $a\in \br {\kappa};{\omega};\cap M_{\alpha}$ then $\cl_{\tau_{\alpha+1}}(a)$ 
is compact. 
\end{enumerate}
We do what we promised in the remark above: we define the topology on larger and larger sets while making the closure of more and more countable sets compact. Condition (\ref{backint}) guarantees that we do not interfere in later steps with the topology  we defined earlier and, in particular, compact sets remain compact. Furthermore, (\ref{enum:Aclosed}) ensures that we deal with all countable sets eventually. In turn, $\mc X_\kappa$ will be the desired splendid space of size $\kappa$.


If $\alpha$ is a limit ordinal then we let $\tau_\alpha$ be the topology generated by $\bigcup\{\tau_\beta:\beta<\alpha\}$. The only property which is not straightforward is that no countable set has uncountable closure suddenly i.e.
\begin{claim}
 $\mc X_\alpha$ is $\oo$-fair.
\end{claim}
\begin{proof}  Suppose that $a\in [M_{<\alpha}\cap \kappa]^\oo$. Fix a countable 
neighborhood  $U_x\in \tau_\alpha$ of $x$ for each $x\in  \cl_{\tau_\alpha}(a)$; 
in turn, $x\in \cl_{\tau_\alpha}(a\cap U_x)$. If $U_x\in N_{{\alpha},i}$, then
$U_x\subs N_{{\alpha},i}$, and so $x\in \cl_{\tau_{\alpha}}(a\cap N_{{\alpha},i})$.

In turn, 
$$\cl_{\tau_\alpha}(a)=\bigcup\{\cl_{\tau_\alpha}(a\cap N_{\alpha,j}):j<\oo\}.$$

However, $b_j=a\cap N_{\alpha,j}\in N_{\alpha,j}\subseteq M_{<\alpha}$ since 
$N_{\alpha,j}$ is countably closed. So $b_j\in M_\xi$ for some $\xi<\alpha$. 
Now, property (\ref{enum:Aclosed}) implies that $B_j=\cl_{\tau_{\xi+1}}(b_j)$ is 
compact and hence countable (indeed, cover $B_j$ by countable open sets and take 
a finite subcover). By property (\ref{backint}), $B_j$ is still a compact set 
containing $b_j$ in $\tau_\alpha$ so $\cl_{\tau_\alpha}(b_j)=B_j$ as well. Hence 
$\cl_{\tau_\alpha}(a)$ is countable.
\end{proof}

Now, we deal with the case of $\alpha=0$, and with the construction of $\mc X_{\alpha+1}$ from $\mc X_\alpha$. This can be done simultaneously. Our main objective is to make the closure of each $a\in \br {\kappa};{\omega};\cap M_{\alpha}$ compact and we shall do that by applying Lemma \ref{lm:ext}. We will show this now while preserving properties (\ref{enum:topprop})-(\ref{enum:Aclosed}). 

First, note that $|X_{\alpha+1}\setm X_{\alpha}|=\aleph_1$ since $M_\alpha\models |M_{<\alpha}|<\kappa$. Second:

\begin{claim}\label{clm:clopenint}
 $X_{\alpha}\cap M_{\alpha}$ is clopen in $\mc X_{\alpha}$.
\end{claim}
\begin{proof}
Indeed, $\tau_\alpha\in M_{\alpha}$ since $\<M_{\zeta}:{\zeta}<{\alpha}\>\in M_{\alpha}$. So if $x\in X_{\alpha}\cap M_{\alpha}$ then $M_\alpha$ contains a (countable) neighbourhood of $x$ which is then a subset of $M_\alpha$ as well. So   $X_{\alpha}\cap M_{\alpha}$ is open.

Second, if $x$ is a $\tau_\alpha$-accumulation point of  $X_{\alpha}\cap M_{\alpha}$ then there is a countable $a\subseteq X_{\alpha}\cap M_{\alpha}$ so that $x\in \cl_{\tau_\alpha}(a)$. As before, $\cl_{\tau_\alpha}(a)\in M_\alpha$ since $a,\tau_\alpha\in M_\alpha$; so  $\cl_{\tau_\alpha}(a)\subs M_\alpha$ since $\cl_{\tau_\alpha}(a)$ is countable. Hence $x\in M_\alpha$. This proves that $X_{\alpha}\cap M_{\alpha}$ is closed.
\end{proof}

Now, let $\mc Y_0$ be  $X_{\alpha}\cap M_{\alpha}$ with the subspace topology from $\tau_\alpha$. Apply Lemma \ref{lm:ext} to find a topology $\rho$ on $Y=\kappa\cap M_\alpha\supseteq Y_0$ which is splendid and satisfies  $\rho_0=\rho\cap \mc P(Y_0)$. Actually, choose $\rho$ to be the $\tri$-minimal topology which satisfies these requirements. Now, we simply let $\tau_{\alpha+1}$ be the topology generated by $\tau_\alpha\cup \rho$; $\tau_{\alpha+1}$ is now clearly reconstructible by the minimal choice of $\rho$.

The only non-trivial condition to verify is that $\mc X_{\alpha+1}$ is $\oo$-fair i.e.  $\cl_{\tau_{\alpha+1}}(a)$ is countable if $a\subseteq X_{\alpha+1}$ is countable. First, note that $X_{\alpha+1}\cap M_\alpha$ is open in $\mc X_{\alpha+1}$. Now, $\cl_{\tau_{\alpha+1}}(a)=\cl_{\tau_{\alpha+1}}(a\setm M_\alpha)\cup \cl_{\tau_{\alpha+1}}(a\cap M_\alpha)$. The set $\cl_{\tau_{\alpha+1}}(a\setm M_\alpha)=\cl_{\tau_\alpha}(a\setm M_\alpha)$ is countable by the inductive hypothesis and the second term is compact (indeed, $a\cap M_\alpha\in [\kappa]^{\le\oo}\cap M_\alpha$) and so countable as well.

This finishes the construction of $\mc X_{\alpha+1}$ and hence the main induction and the proof of the theorem.

\end{proof}

The take-away from this proof (especially compared with the original from \cite{juhawe}) should be that sage Davies-trees allowed us to simply lift the base-case of our construction of splendid spaces (i.e. Lemma \ref{lm:ext}) to construct arbitrary large splendid spaces.

Our last remark is that a splendid space can be used to produce cofinal Kurepa-families \cite{good}, and so CH and the existence of sage Davies-trees for $\kappa$ implies the existence of a cofinal Kurepa family in $\br \kappa; \oo;$.

\section{Appendix: how to construct sage Davies-trees?}\label{sec:appendix}

\newcommand{\N}{K}
\newcommand{\kk}{K}
\newcommand{\mm}{M}
\newcommand{\mmcal}{{\mc M}}
\newcommand{\nn}{N}
\newcommand{\nncal}{{\mc N}}
\newcommand{\sm}{{L}}
\newcommand{\sicc}{$\sigma$-c.c}

\newcommand{\kkk}{K^*}

%
%
%
%
%

The goal of this section is to prove Theorem \ref{tm:niceomegadavies}, that is 
the existence of sage Davies-trees for $\kappa$. 
{The main  difficulty of the proof 
comes from the fact that a large countably 
closed model can not be written as a continuous chain of smaller, still countably 
closed models in general. The fact that we need to show that  
$\<M_{\alpha}:{\alpha}<{\beta}\>\in M_\beta$ holds adds multiple layers 
of difficulties compared to constructing only high Davies-trees. In particular, we needed to abandon the original tree-construction and take a completely different road; this will be somewhat reminiscent of the \emph{dominating matrices} from \cite{freese}. Whenever we say or write \emph{submodel} in this section, we will mean an elementary submodel of $(H(\theta),\in, \tri)$ where $\tri$ is a well order of $H(\theta)$.

\smallskip


\smallskip

We start by a definition:

\begin{definition}
 We say that $\sm\prec H(\theta)$ is a {\em sage model for ${\kappa}$ over $x$} if
there is a sage Davies-tree $\<M_{\alpha}:{\alpha}<{\kappa}\>$ for ${\kappa}$ over $x$ 
 such that $\sm=\bigcup_{{\alpha}<{\kappa}}M_{\alpha}$.
 
 \smallskip
 \end{definition}
 
In particular, a sage model is countably closed. If $L$ is a sage model for ${\kappa}$ over $x$ then we fix an arbitrary sage Davies-tree $\<M_{\alpha}:{\alpha}<{\kappa}\>$ with union $L$ and write
\begin{displaymath}
\mathcal{M}(\sm)=\<\mm_{\alpha}:{\alpha}<{\kappa}\>\quad\text{ and }\quad
\mathcal{M}(\sm,{\alpha})=\mm_{\alpha},
\end{displaymath}
and
\begin{displaymath}
\nncal(\sm,{\beta},j)=\nn_{{\beta},j}
\end{displaymath}
where   $\nn_{\beta,j}\prec H(\theta)$ with $[\nn_{\beta,j}]^\oo\subs \nn_{\beta,j}$ and 
$x\in \nn_{\beta,j}$ for $j<\oo$ such that $$\bigcup\{\mm_{\alpha}:\alpha<\beta\}=\bigcup\{\nn_{\beta,j}:j<\oo\}.$$

\smallskip

So Theorem  \ref{tm:niceomegadavies} is equivalent to the following:

\begin{theorem}\label{tm:sage-model}
There is a sage model $\sm({\kappa},x)$  for ${\kappa}$ over $x$ whenever
\begin{enumerate}
 \item ${\kappa}={\kappa}^{\omega}$, and
 \item ${\mu}$ is ${\omega}$-inaccessible, ${{\mu}}^{\omega}={\mu}^+$ and
 $\Box^{}_{\mu}$ holds for all  $\mu$ with $\mf c< {\mu}<{\kappa}$ and  $\cf({\mu})={\omega}$.
 \end{enumerate}
\end{theorem}


\begin{proof}
We prove the theorem by induction on ${\kappa}$: the base case of  ${\kappa}=\mf c$ is trivial, so suppose that the statement is true for $\mf c\leq {\lambda}<{\kappa}$. If ${\lambda}={\lambda}^{\omega}<{\kappa}$ and $x\in H(\theta)$ is a parameter let $\sm({\lambda},x)$ denote a sage model of size ${\lambda}$.

%

There will be two cases we need to consider: either $\kappa$ is ${\omega}$-inaccessible or ${\kappa}={\mu}^+$ for some ${\mu}>\mf c$ with $\cf({\mu})={\omega}$. Indeed, if $\mu<\kappa$ is minimal such that $\mu^\oo=\kappa$ then $\cf (\mu)=\oo$. Otherwise,  $\kappa=\mu^\oo=\sup_{i<\cf(\mu)}\mu_i^\oo$ where $(\mu_i)_{i<\cf(\mu)}$ is any cofinal sequence in $\mu$. So, $\mu\leq \mu_i^\oo$ for some $i<\cf(\mu)$ and hence $\kappa=\mu^\oo\leq (\mu_i^\oo)^\oo=\mu_i^\oo<\kappa$, a contradiction.

Let us say that  $X\subs H(\theta)$ is {\em \sicc} if $X$ is the union of countably many countably 
closed elementary submodels of $H(\theta)$. Now, we deal with the above described two cases:

\medskip
\noindent {\bf Case I.}  {\em ${\kappa}$ is ${\omega}$-inaccessible.}
\medskip

Enumerate $[\kappa]^\oo$ as 
$\{y_\alpha:\alpha<\kappa\}$. By induction on $\alpha<\cf(\kappa)$ we define an increasing sequence of models 
$\<\kk_{\alpha}:{\alpha}\le {\kappa}\>$
such that 
$|\kk_{\alpha}|=(\max(\mf c,|{\alpha}|))^{\omega}$
and that 
$\kk_{{\alpha}+1}$ is a sage model for each ${\alpha}<{\kappa}$
as follows.

For $\alpha=0$, we just find any countably closed model $\kk_0$ of size $\mf c$ with 
$x\in \kk_0$.

Now, if $\alpha$ is a limit ordinal then we let 
$\kk_\alpha=\bigcup\{\kk_{\beta}:{\beta}<\alpha\}$.

Finally, given $\kk_\alpha$ we define  $\kk_{\alpha+1}$: let 
\begin{displaymath}
x_{{\alpha}+1}=\{x,y_{\alpha}, \<\mmcal(\kk_{{\gamma}+1},{\zeta}): 
{\gamma}+1\le {\alpha,{\zeta}}<|\kk_{{\gamma}+1}|\>\} 
\end{displaymath}
and
\begin{displaymath}
 \kk_{{\alpha}+1}=\sm((\max(\mf c,|{\alpha}+1|))^{\omega}, x_{{\alpha}+1})
\end{displaymath}

\begin{claimn}
  $\kk_{\alpha}$ is countably closed unless $\cf({\alpha})={\omega}$.
\end{claimn}
 
If ${\alpha}={\beta}+1$ or ${\alpha}=0$, then the statement is trivial because $\kk_{\alpha}$  
is a sage model.

Assume  $\cf(\alpha)>\oo$. If $a\in [\kk_\alpha]^\oo$ then there is 
$\zeta<\alpha$ so that $a\in [\kk_\zeta]^\oo$. So $a\in [\kk_{\zeta+1}]^\oo$ since 
$\kk_\zeta\subseteq \kk_{\zeta+1}$. The model $\kk_{\zeta+1}$ is countably closed so $a\in 
\kk_{\zeta+1}\subseteq \kk_\alpha$ as desired. 

\medskip

\begin{claimn}
 If $\cf(\alpha)=\oo$ then 
$\kk_{\alpha}$ is \sicc.\end{claimn}

Indeed,  take any cofinal $\oo$-sequence $\<\alpha_k\>_{k<\oo}$ 
in $\alpha$.
Then $\kk_{\alpha}=\bigcup_{k<{\omega}}\kk_{{\alpha}_k+1}$,
and every $\kk_{{\alpha}_k+1}$ is countably closed.

\medskip

Now, our goal is to show that $L=\kk_{\kappa}$ is a sage model; informally,  the concatenation of the sequences
$\<\mc M(\kk_{{\alpha}+1}):{\alpha}<{\kappa}\>$ will witness this. More precisely, let $\Big\langle\<{\alpha}_\eta,{\zeta}_\eta\>:\eta<{\kappa}\Big\rangle$
be the lexicographically increasing enumeration of the set 
\begin{displaymath}
 \{
 \<{\alpha}+1,{\zeta}\>:{\alpha}+1<{\kappa},{\zeta}<|\kk_{{\alpha}+1}|
 \},
\end{displaymath}
and write
\begin{displaymath}
 \mm'_\eta=\mmcal(\kk_{{\alpha}_\eta},\zeta_\eta)
\end{displaymath}
for $\eta<{\kappa}$.


We will show that the  sequence $\<\mm'_{\zeta}:{\zeta}<{\kappa}\>$
witnesses that $L=\kk_{\kappa}$ is a sage model.

First, it is clear that
$$L=\bigcup_{{\alpha}+1<{\kappa}}\kk_{{\alpha}+1}=
\bigcup_{{\alpha}+1<{\kappa}}\Big(\bigcup_{\zeta<|\kk_{{\alpha}+1}|}\mmcal(\kk_{{\alpha}+1},\zeta)\Big)=
\bigcup_{\eta<{\kappa}}\mm'_\eta.$$

Furthermore, $L$ is a countably closed model by the first claim above.

Second, we need to show that $\<\mm'_\eta:\eta<{\kappa}\>$ is a sage Davies-tree; first, we prove that $\<\mm'_\eta:\eta<{\kappa}\>$ is a high Davies-tree.

If ${\eta}<{\kappa}$ and  ${\alpha}_{\eta}={\beta}+1$ then
\begin{align*}
 \bigcup_{\sigma<\eta}\mm'_\sigma
 =
 \bigcup\{\mm'_\sigma: {\alpha}_\sigma<{\alpha}_\eta\}
 \cup\bigcup\{\mm'_\sigma: {\alpha}_\sigma={\alpha}_\eta\ \land\ \zeta_\sigma<\zeta_\eta\}
 =\\
 \bigcup\{\mm'_\sigma: {\alpha}_\sigma\le {\beta}\}
 \cup\bigcup\{\mm'_\sigma: {\alpha}_\sigma
 ={\beta}+1\ \land\ \zeta_\sigma<\zeta_\eta\}=\\
 \bigcup \{\kk_{{\alpha}+1}:{\alpha}+1\le {\beta}\}\cup 
 \bigcup\{\mmcal(\kk_{{\beta}+1},\zeta):\zeta<\zeta_\eta\}
 =\\
 \kk_{{\beta}}\cup \bigcup\{\nncal(\kk_{{\beta}+1},\zeta_\eta,j):j<{\omega}\}.
\end{align*}
Since $\kk_{\beta}$ is \sicc\ by the claims and every $\nncal(\kk_{{\beta}+1},\zeta_\eta,j)$
is countably closed, we proved that $\bigcup_{\sigma<\eta}\mm'_\sigma$
 is also \sicc. Also, the parameter $x$ is contained in all the models. Finally, using the sets $y_\alpha$ we made sure that $L$ covers $[\kappa]^\oo$. So $\<\mm'_\eta:\eta<{\kappa}\>$ really is a high Davies-tree.

To  show that $\<\mm'_\sigma:\sigma<\eta\>\in \mm'_\eta$ observe that 
if ${\alpha}_{\eta}={\beta}+1$, then 
\begin{align}\label{eq:cut}
\<\mm'_\sigma:\sigma<\eta\>
=\<\mm'_\sigma:{\alpha}_\sigma<{\alpha}_\eta\>^\frown
&\<\mm'_\sigma:{\alpha}_\sigma={\alpha}_\eta\
 \land\ \zeta_\sigma<\zeta_\eta\>=\\\notag
 \<\mm'_\sigma:{\alpha}_\sigma\le{\beta}\>^\frown
&\<\mm'_\sigma:{\alpha}_\sigma={\beta}+1\
 \land\ \zeta_\sigma<\zeta_\eta\>.
\end{align}

Now $\<\mm'_\sigma:{\alpha}_\sigma\le{\beta}\>$ is the lexicographical enumeration of
$$\mf M_{\beta}=\<\mmcal(\kk_{\gamma+1},\zeta ):{\gamma}+1\le {\beta}, {\zeta}<|\kk_{\gamma+1}|\>.$$
Since $\mf M_{\beta}\in x_{{\beta}+1}\subs \mm (\kk_{{\beta}+1}, {\zeta}_{\eta})=
\mm'_{\eta}$ 
we have
\begin{equation}\label{eq:param1}
 \<\mm_\sigma:{\alpha}_\sigma\le{\beta}\>\in \mm'_{\eta}.
\end{equation}
Finally, since $\mmcal(\kk_{{\beta}+1})$ is a sage Davies-tree, we have 
\begin{equation}\label{eq:sage}
 \<\mmcal(\kk_{{\beta}+1},{\zeta}):{\zeta}<{\zeta}_{\eta}\>\in  \mmcal(\kk_{{\beta}+1},{\zeta}_{\eta})=
 \mm'_{\eta}.
\end{equation}
Now \eqref{eq:cut}, \eqref{eq:param1}
and \eqref{eq:sage} give that $\<\mm'_\sigma:\sigma<\eta\>\in \mm'_\eta$. 

This finishes the proof of $\<\mm'_\eta:\eta<{\kappa}\>$ being a sage Davies-tree and hence the proof of Case I.

\bigskip

\noindent{\bf Case II.}  {\em ${\kappa}={\mu}^+$ for some ${\mu}>\mf c$ with $cf({\mu})={\omega}$.}

\medskip
Let $\<C_{\alpha}:{\alpha}<{\mu}^+\>$ witness that
$\Box_{{\mu}}$ holds  and
fix an increasing  sequence of regular cardinals $\<{\mu}_j:j<{\omega}\>$
with ${{\mu}_j}^{\omega}={\mu}_j$ and ${\mu}=\sup_{j<\oo} {\mu}_j$ (e.g. 
take any cofinal sequence $\<\nu_j:j\in\oo\>$ in $\mu$ and let $\mu_j=({\nu_j}^\oo)^+$). 
Also, let $\br {\kappa};{\omega};=\{y_{\alpha}:{\alpha}<{\kappa}\}$.


The plan is to define a matrix of elementary submodels with $\kappa$ rows and 
$\oo$ columns so that the rows union up to the desired sage model. 
Unfortunately, we need some technical assumptions to carry out this 
construction. In more detail, we will construct elementary submodels 
 $\<\N_{{\alpha},j}:{\alpha}<{\kappa},j<{\omega}\>$  of $H(\theta)$ by induction 
on
$\alpha<\kappa$  such that writing $\kk_{\alpha}=\bigcup\limits_{j\in{\omega}} 
\N_{\alpha,j}$, \smallskip
properties 
(A)--(\ref{pr:last}) below hold:
\begin{enumerate}[(A)]
 \item\label{Nst-elem}\label{pr:first}  $\N_{{\alpha},j}\prec H(\theta)$, 
$|\N_{{\alpha},j}|= {\mu}_j$ and ${\mu}_j+1\subs \N_{{\alpha},j}$,  \smallskip
\item\label{Nst-monk}  $\N_{\alpha,j}\subs \N_{\alpha,j+1}$ and $\N_{\alpha,j}\subs \N_{\alpha+1,j}$,\smallskip
\item   \label{Nst-incr}   $\forall {\beta}<{\alpha}$ $\exists k_{{\beta},{\alpha}}<\oo$ such that 
$\N_{{\beta},j}\subs \N_{{\alpha},j}$ for all $ j\ge k_{{\beta},{\alpha}}$, \smallskip
\item \label{pr:extrahomog}
if $|C_{\alpha}|\le {\mu}_j$ and ${\gamma}\in C_{\alpha}'$, then $\N_{{\gamma}+1,j}\subs \N_{{\alpha},j}$,\smallskip
\item    \label{pr:manyparam} $\N_{\alpha+1,j}$ is a sage model defined as follows: let 
\begin{multline*}
x_{{\alpha}+1,j}=\{x,y_{\alpha},\<\mmcal(\N_{{\gamma}+1,j}):{\gamma}+1\le {\alpha},j<{\omega},
{\zeta}<|\N_{{\gamma}+1,j}|\>,\\
\<\mmcal(\N_{{\alpha}+1,i},{\zeta}):i<j,{\zeta}<|\N_{{\alpha}+1,i}|\>\}  
 \end{multline*}
and define
$$\N_{{\alpha}+1,j}=
\sm({\mu}_{j}, x_{{\alpha}+1,j}
),$$
\item \label{pr:manyomegaclosed}

$\br \N_{{\alpha},j};{\omega};\subs \N_{{\alpha},j} $ unless $\cf({\alpha})={\omega}$
and ${\alpha}=\sup C_{\alpha}'$,\smallskip
\item \label{pr:manyallclosed} 

$\kk_{\alpha}$ is \sicc. and

\item \label{pr:last}\label{pr:Ncont}
if ${\alpha}<{\kappa}$ is a limit ordinal then $\kk_{\alpha}=\bigcup_{{\beta}<{\alpha}}\kk_{\beta}$. \smallskip
\end{enumerate}

Here, $C_{\alpha}'$ denotes the set of accumulation points of $C_\alpha$ i.e. $\gamma\in C'_\alpha$ iff $\gamma=\sup(C_\alpha\cap \gamma)$. 

 \begin{figure}[H]
 \centering
\scalebox{1} 
{
\begin{pspicture}(0,-3.1789062)(11.202812,3.1589062)
\psline[linewidth=0.04cm,arrowsize=0.05291667cm 2.0,arrowlength=1.4,arrowinset=0.4]{->}(1.1609375,-2.5610938)(1.1809375,3.1389062)
\psline[linewidth=0.03cm,linestyle=dashed,dash=0.16cm 0.16cm,arrowsize=0.05291667cm 2.0,arrowlength=1.4,arrowinset=0.4]{->}(0.9809375,-2.2610939)(7.1809373,-2.2610939)
\psline[linewidth=0.03cm,linestyle=dashed,dash=0.16cm 0.16cm,arrowsize=0.05291667cm 2.0,arrowlength=1.4,arrowinset=0.4]{->}(0.9809375,0.13890626)(7.1809373,0.13890626)
\psline[linewidth=0.03cm,linestyle=dashed,dash=0.16cm 0.16cm,arrowsize=0.05291667cm 2.0,arrowlength=1.4,arrowinset=0.4]{->}(0.9809375,1.7389063)(7.1809373,1.7389063)
\psdots[dotsize=0.2](1.9809375,-2.2610939)
\psdots[dotsize=0.2](7.7809377,-2.2610939)
\psdots[dotsize=0.2](1.9809375,0.13890626)
\psdots[dotsize=0.2](1.9809375,1.7389063)
\psdots[dotsize=0.2](4.3809376,-2.2610939)
\psdots[dotsize=0.2](4.3809376,0.13890626)
\psdots[dotsize=0.2](4.3809376,1.7389063)
\psdots[dotsize=0.2](5.7809377,1.7389063)
\psdots[dotsize=0.2](5.7809377,0.13890626)
\psdots[dotsize=0.2](5.7809377,-2.2610939)
\psdots[dotsize=0.2](7.7809377,0.13890626)
\psdots[dotsize=0.2](7.7809377,1.7389063)
\usefont{T1}{ptm}{m}{n}
\rput(0.6223438,2.8439062){$\kappa$}
\usefont{T1}{ptm}{m}{n}
\rput(0.5,0.14390624){$\alpha$}
\usefont{T1}{ptm}{m}{n}
\rput(0.4,1.7439063){$\alpha+1$}
\usefont{T1}{ptm}{m}{n}
\rput(2.05,-0.55609375){$\N_{\alpha,0}$}
\usefont{T1}{ptm}{m}{n}
\rput(2.05,2.4439063){$\N_{\alpha+1,0}$}
\usefont{T1}{ptm}{m}{n}
\rput(5.4,2.4439063){$\N_{\alpha+1,j}\in \N_{\alpha+1,j+1}$}
\usefont{T1}{ptm}{m}{n}
\usefont{T1}{ptm}{m}{n}
\rput(5.4,-0.55609375){$\N_{\alpha,j}\hspace{0.15cm}\in\hspace{0.2cm}\N_{\alpha,j+1}$}
\rput(4.4,0.88609375){$\vsubs$}
\rput(5.78,0.88609375){$\vsubs$}
\usefont{T1}{ptm}{m}{n}
\usefont{T1}{ptm}{m}{n}
\rput(8.962344,1.7439063){$\kk_{\alpha+1}$}
\usefont{T1}{ptm}{m}{n}
\rput(8.772344,0.14390624){$\kk_{\alpha}$}
\usefont{T1}{ptm}{m}{n}
\rput(8.772344,-2.0560938){$\kk_{0}$}
\usefont{T1}{ptm}{m}{n}
\rput(2.05,-2.9560938){$\N_{0,0}$}
\usefont{T1}{ptm}{m}{n}
\rput(5.4,-2.9560938){$\N_{0,j}\hspace{0.15cm}\in\hspace{0.2cm}\N_{0,j+1}$}
\usefont{T1}{ptm}{m}{n}
\end{pspicture} 
}
\caption{The matrix of models $\<\N_{{\alpha},j}:{\alpha}<{\kappa},j<{\omega}\>$}
\label{fig:1}

\end{figure}
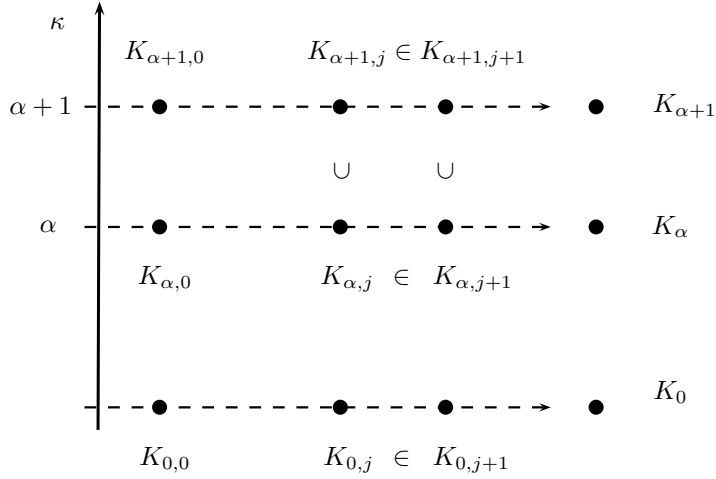

\medskip

\begin{claim}\label{clm:sageseq}
 $\kk=\bigcup\{\kk_\alpha:\alpha<\kappa\}$ is a sage model for $\kappa$ over $x$.
\end{claim}

\begin{proof}
 Since   $\N_{{\alpha},j}\prec H(\theta)$  and $\N_{{\alpha},j}\subs \N_{{\alpha},j+1}$ by (\ref{Nst-elem})
and (\ref{Nst-monk}) we have $\kk_{\alpha}\prec H(\theta)$. 

Since $\kk_{\alpha}\subs \kk_{\beta}$ for ${\alpha}<{\beta}<{\kappa}$ by  (\ref{Nst-incr}),
we have $\kk_{\alpha}\prec\kk_{\beta}$ for ${\alpha}<{\beta}<{\kappa}$, and so $\kk\prec H(\theta)$.

Next observe that $\br {\kappa};{\omega};\subs \kk$ because $y_{\alpha}\in \kk_{{\alpha}+1}$
by  (\ref{pr:manyparam}).

To show $\br K;{\omega};\subs K$ assume that $A\in \br K;{\omega};$.
Since ${\kappa}^{\omega}={\kappa}$ implies $\cf({\kappa})>{\omega}$, there is ${\alpha}<{\kappa}$
such that $A\subs \kk_{\alpha}$.

Then, by (\ref{pr:manyparam}),  $x_{{\alpha}+1,0}\in \kk_{{\alpha}+1}$,
and so $\kk_{\alpha}\in \kk_{{\alpha}+1}\subs \kk.$
Since $|\kk_{\alpha}|={\mu}$ and ${\mu}+1\subs \kk$, there is a bijection $f:\kk_{\alpha}\to {\mu}$
in $\kk$.  Let $Y=f''A$. Then $Y\in \br {\mu};{\omega};\subs \br{\kappa};{\omega};\subs \kk$.
Thus $A=f^{-1}Y\in \kk$ as well.

So $\kk$ is a countably closed elementary submodel of $H(\theta)$.

Let $\Big\langle\<{\alpha}_{\eta},j_{\eta},{\zeta}_{\eta}\>:{\eta}<{\kappa}\Big\rangle$
be the lexicographically increasing enumeration of the set 
\begin{displaymath}
 \Big\{
 \<{\alpha}+1,j,{\zeta}\>:{\alpha}+1<{\kappa},j<{\omega},{\zeta}<|\kk_{{\alpha}+1,j}|
 \Big\},
\end{displaymath}
and write
\begin{displaymath}
 \mm'_\eta=\mmcal(\kk_{{\alpha}_\eta,j_{\eta}},\zeta_\eta)
\end{displaymath}
for $\eta<{\kappa}$.

Clearly 
\begin{align*}
\kk=\bigcup_{{\alpha}<{\kappa}}\kk_{{\alpha}+1}=
\bigcup_{{\alpha}<{\kappa}}\bigcup_{j<{\omega}}\kk_{{\alpha}+1,j}=&\\
\bigcup_{{\alpha}<{\kappa}}\bigcup_{j<{\omega}}\bigcup_{{\zeta}<|\kk_{{\alpha}+1,j}|}
&\mmcal(\kk_{{\alpha}+1,j},{\zeta})=\bigcup_{{\eta}<{\kappa}}\mm'_{\eta}.
\end{align*}

We need to show that $\<\mm'_\eta:\eta<{\kappa}\>$ is a sage Davies-tree.

If ${\eta}<{\kappa}$ and ${\alpha}_{\eta}={\beta}+1$, then 
\begin{multline*}
 \bigcup_{\sigma<\eta}\mm'_\sigma=\bigcup\{\mm'_\sigma: {\alpha}_\sigma<{\alpha}_\eta\}
 \cup\bigcup\{\mm'_\sigma: {\alpha}_\sigma={\alpha}_\eta\land j_\sigma<j_\eta\}\\
 \cup\bigcup\{\mm'_\sigma: {\alpha}_\sigma={\alpha}_\eta\land j_\sigma=j_\eta\land \zeta_\sigma<\zeta_\eta\}=
 \\
 \bigcup \{\kk_{{\alpha}+1}:{\alpha}+1\le {\beta}\}\cup
 \bigcup \{\N_{{\beta}+1,j}:j<j_{\eta}\}\cup
 \bigcup\{\mmcal(\N_{{\beta}+1,j_\eta},\zeta):\zeta<\zeta_\eta\}=\\
 \kk_{{\beta}}\cup \bigcup \{\N_{{\beta}+1,j}:j<j_{\eta}\}\cup 
 \bigcup\{\nncal(\N_{{\beta}+1,j_\eta},\zeta_\eta,i):i<{\omega}\}.
\end{multline*}
Since $\kk_{{\beta}}$ is \sicc\ by (\ref{pr:manyallclosed}), 
 every $\kk_{{\beta}+1,j}$  and 
every $\nncal(\kk_{{\alpha}_\eta,j_{\eta}},\zeta_\eta,i)$
is countably closed, we have that $\bigcup_{\sigma<\eta}\mm'_\sigma$
 is also \sicc. Furthermore, the parameter $x$ is in all the above models and $[\kappa]^\oo\subseteq \bigcup\{M'_\eta:\eta<\kappa\}$ by (\ref{pr:manyparam}); so $\<M'_\eta:\eta<\kappa\>$ is a high Davies-tree for $\kappa$ over $x$.

 To  show that $\<\mm'_\sigma:\sigma<\eta\>\in \mm'_\eta$ observe that 
if ${\alpha}_{\eta}={\beta}+1$, then 
\begin{align}\label{eq:cut2}
\<\mm'_\sigma:\sigma<\eta\>=v_0^\frown v_1^\frown v_2
\end{align}
where
\begin{align}
v_0=&\<\mm'_\sigma:{\alpha}_\sigma\le {\beta}\>,\\ 
v_1=&\<\mm'_\sigma:{\alpha}_\sigma={\beta}+1, j_{\sigma}<j_{\eta}\>=
\mmcal(\N_{{\beta}+1,0})^\frown\dots{}^\frown\mmcal(\N_{{\beta}+1,j_\eta-1}), \textmd{ and}\\ 
v_2=&\<\mm'_\sigma:{\alpha}_\sigma={\beta}+1,j_{\sigma}=j_{\eta},{\zeta}_{\sigma}<{\zeta}_{\eta}\>=
\<\mmcal(\N_{{\alpha}_{\eta},j_{\eta}},\zeta):\zeta<{\zeta}_{\rho}\>.
\end{align}

First, $v_0=\<\mm'_\sigma:{\alpha}_\sigma\le{\beta}\>$ is the lexicographical enumeration of
$$\mf M_{\beta}=\<\mmcal(\kk_{\gamma+1,i},\zeta ):{\gamma}+1\le {\beta},i<{\omega}, {\zeta}<|\kk_{\gamma+1}|\>.$$
Now, $\mf M_{\beta} \in x_{{\beta}+1,0}\in 
\mmcal (\N_{{\beta}+1,j_{\eta}}, {\zeta}_{\eta})=
\mm'_{\eta}$ implies 
$\mf M_{\beta}\in \mm'_{\eta}$, so
we have
\begin{equation}\label{eq:param2}
 v_0=\<\mm'_\sigma:{\alpha}_\sigma\le{\beta}\>\in \mm'_{\eta}.
\end{equation}


Second, 
$\<\mmcal(\N_{{\beta}+1,i},{\zeta}):i<j_\eta,{\zeta}<|\N_{{\beta}+1,i}|\>
\in {x_{\beta+1,j_{\eta}}}\in \mmcal (\N_{{\beta}+1,j_{\eta}},\zeta_\eta)=\mm'_{\eta}$ by  (\ref{pr:manyparam}), and so 
\begin{equation}
 v_1=\mmcal(\N_{{\beta}+1,0})^\frown\dots{}^\frown\mmcal(\N_{{\beta}+1,j_\eta-1})\in \mm'_{\eta}.
\end{equation}
%
%

Since $\mmcal(\N_{{\beta}+1,j_{\eta}})$ is a sage  Davies-tree, we have 
\begin{equation}\label{eq:sage2}
 v_2=\<\mmcal(\N_{{\beta}+1,j_{\eta}},{\zeta}):{\zeta}<{\zeta}_{\eta}\>\in  
 \mmcal(\kk_{{\beta}+1,j_{\eta}},{\zeta}_{\eta})=
 \mm'_{\eta}.
\end{equation}

Thus $v_0,v_1,v_2\in \mm'_{\eta}$ and so 
$\<\mm'_\sigma:\sigma<\eta\>=v_0^\frown v_1^\frown v_2\in \mm'_{\eta}$ as required. This proves that $\<M'_\eta:\eta<\kappa\>$ is a sage Davies-tree for $\kappa$ over $x$.

\end{proof}

So, our job is to build the matrix $\<\N_{{\alpha},j}:{\alpha}<{\kappa},j<{\omega}\>$ with the above properties (one mainly needs to keep Figure \ref{fig:1} in mind). The induction is naturally divided into four cases, the first two being easier and the second two a bit more involved. However, the proofs are fairly straightforward diagram and definition chasings so we might leave some details to the reader.
\medskip

\textbf{Case 1: $\alpha=0$.} Let $\N_{0,j}=\sm({\mu}_j,\{{\mu}_j+1,x\})$ using the inductive assumption for $\mu_j<\kappa$.  Then  $|\N_{0,j}|={\mu}_j$ and properties (A)--(\ref{pr:last}) are easily checked.

\medskip

\textbf{Case 2: successor steps from $\alpha$ to $\alpha+1$.} We aim to define $\<\N_{\alpha+1,j}:j<\oo\>$.  

We follow (\ref{pr:manyparam}):
writing 
\begin{multline*}
x_{{\alpha}+1,j}=\{x,y_{\alpha},\<\mmcal(\N_{{\gamma}+1,j}):{\gamma}+1\le {\alpha},j<{\omega},
{\zeta}<|\N_{{\gamma}+1,j}|\>,\\
\<\mmcal(\N_{{\alpha}+1,i},{\zeta}):i<j,{\zeta}<|\N_{{\alpha}+1,i}|\>\}  
 \end{multline*}
we take 
$$\N_{{\alpha}+1,j}=
\sm({\mu}_{j}, x_{{\alpha}+1,j}).$$

Again, properties (A)--(\ref{pr:last}) are easy to verify.

\medskip

\textbf{Case 3: $\alpha$ is a limit and $\sup C'_\alpha<\alpha$.} Then $C_{\alpha}\setm \sup {C_{\alpha}}'$ must have order type ${\omega}$, so  
$C_{\alpha}\setm \sup {C_{\alpha}}'$ can be enumerated 
as $${\beta}_0<{\beta}_1<\dots.$$
Write ${\gamma}_i={\beta}_i+1$ and fix a strictly increasing sequence of natural numbers $k_0<k_1<\dots$ such that 
\begin{displaymath}
  \N_{{\gamma_i},j}\subs \N_{{\gamma_{i+1}},j} \textmd{ for all }j\ge k_i. 
\end{displaymath}

This is possible by applying (\ref{Nst-incr}).  Now let 
\begin{equation}\label{eq:Naknotsup}
\N_{{\alpha},j}=\left\{
\begin{array}{ll}
\N_{\gamma_0,j}&\mbox{if $j<k_0,$ and}\\\\ 
\N_{{\gamma_{i+1}},j}&\mbox{if $k_i\le j<k_{i+1}$.}
\end{array}
\right . 
\end{equation}

In other words, we take finite intervals from the sequences $\<\N_{\gamma_i,j}:j<\oo\>$ to form the $\alpha^{th}$ sequence; see Figure \ref{fig:2}.
  \begin{figure}[H]

   \centering
%
\scalebox{1} 
{
\begin{pspicture}(0,-3.2789063)(13.565157,3.2589064)
\psline[linewidth=0.04cm,arrowsize=0.05291667cm 2.0,arrowlength=1.4,arrowinset=0.4]{->}(1.541875,-1.8810937)(1.56,3.2389061)
\psdots[dotsize=0.2](2.561875,-1.5810939)
\psdots[dotsize=0.2](3.361875,-1.5810939)
\psdots[dotsize=0.2](4.161875,-1.5810939)
\psframe[linewidth=0.04,linestyle=dashed,dash=0.16cm 0.16cm,dimen=outer](4.56,-1.1610937)(2.16,-1.9610938)
\psdots[dotsize=0.2](10.561875,-1.5810939)
\psdots[dotsize=0.2](10.561875,0.8189063)
\psdots[dotsize=0.2](10.561875,2.4189065)
\usefont{T1}{ptm}{m}{n}
\rput(11.534688,2.4239063){$\kk_{\alpha}$}
\usefont{T1}{ptm}{m}{n}
\rput(11.534688,0.8239063){$\kk_{\gamma_1}$}
\usefont{T1}{ptm}{m}{n}
\rput(11.534688,-1.60937){$\kk_{\gamma_0}$}
\psdots[dotsize=0.2](5.961875,-1.5810939)
\psdots[dotsize=0.2](7.361875,-1.5810939)
\psdots[dotsize=0.2](5.961875,0.8189063)
\psdots[dotsize=0.2](7.361875,0.8189063)
\psdots[dotsize=0.2](5.961875,2.4189065)
\psdots[dotsize=0.2](7.361875,2.4189065)
\psframe[linewidth=0.04,linestyle=dashed,dash=0.16cm 0.16cm,dimen=outer](7.76,1.2389063)(5.56,0.43890625)
\psframe[linewidth=0.04,linestyle=dashed,dash=0.16cm 0.16cm,dimen=outer](7.76,2.8389063)(5.56,2.0389063)
\psbezier[linewidth=0.02,linestyle=dashed,dash=0.16cm 0.16cm,arrowsize=0.05291667cm 2.0,arrowlength=1.4,arrowinset=0.4]{->}(7.96,0.93312186)(8.131429,1.0194633)(8.2,1.2969722)(8.2,1.5255089)(8.2,1.7540455)(8.131429,1.881171)(7.96,1.9989063)
\usefont{T1}{ptm}{m}{n}
\rput(0.8,0.8239063){$\gamma_1$}
\usefont{T1}{ptm}{m}{n}
\rput(0.8,2.4239063){$\alpha$}
\psdots[dotsize=0.2](2.561875,2.4189062)
\psdots[dotsize=0.2](3.361875,2.4189062)
\psdots[dotsize=0.2](4.161875,2.4189062)
\psframe[linewidth=0.04,linestyle=dashed,dash=0.16cm 0.16cm,dimen=outer](4.56,2.8389063)(2.16,2.0389063)
\psbezier[linewidth=0.02,linestyle=dashed,dash=0.16cm 0.16cm,arrowsize=0.05291667cm 2.0,arrowlength=1.4,arrowinset=0.4]{->}(4.753684,-1.3010937)(5.02,-0.50677115)(5.0277896,-0.13469376)(5.0538945,0.21090625)(5.08,0.5565063)(4.9829564,1.5069063)(4.741818,1.9389062)
\usefont{T1}{ptm}{m}{n}
\rput(0.8,-1.5760937){$\gamma_0$}
\psline[linewidth=0.04cm,arrowsize=0.05291667cm 2.0,arrowlength=1.4,arrowinset=0.4]{->}(1.76,-2.3610938)(9.96,-2.3610938)
\psline[linewidth=0.04cm](5.96,-2.1610937)(5.96,-2.5610938)
\psline[linewidth=0.04cm](9.16,-2.1610937)(9.16,-2.5610938)
\usefont{T1}{ptm}{m}{n}
\rput(6,-3.0560937){$k_0$}
\usefont{T1}{ptm}{m}{n}
\rput(9.191406,-3.0560937){$k_1$}
\psline[linewidth=0.04,linestyle=dashed,dash=0.16cm 0.16cm](9.76,2.8389063)(8.76,2.8389063)(8.76,2.0389063)(9.76,2.0389063)
\psdots[dotsize=0.2](9.161875,2.4189065)
\usefont{T1}{ptm}{m}{n}
\rput(5.9,-0.33609375){$\vsubs$}
\usefont{T1}{ptm}{m}{n}
\rput(7.35,-0.31609374){$\vsubs$}
\usefont{T1}{ptm}{m}{n}
\rput(6.6,-1.5560937){$\subs$}
\usefont{T1}{ptm}{m}{n}
\rput(2.95,-1.5760938){$\subs$}
\usefont{T1}{ptm}{m}{n}
\rput(3.75,-1.5560937){$\subs$}
\usefont{T1}{ptm}{m}{n}
\rput(5.1114063,-1.5560937){$\subs$}
\usefont{T1}{ptm}{m}{n}
\rput(6.6714063,0.82390624){$\subs$}
\psline[linewidth=0.04cm](1.36,0.8389062)(1.76,0.8389062)
\psline[linewidth=0.04cm](1.36,-1.5610938)(1.76,-1.5610938)
\psline[linewidth=0.04cm](1.36,2.4389062)(1.76,2.4389062)
\end{pspicture} }

  \caption{Case 3 and the construction of $\<\N_{\alpha,j}:j<\oo\>$}
\label{fig:2}

\end{figure}
 Let us go through properties (A)--(\ref{pr:last}) now.

Property (\ref{Nst-elem}) holds since $\N_{{\alpha},j}=\N_{{\beta}+1,j}$ for some ${\beta}\in C_{\alpha}$.

\medskip
To check (\ref{Nst-monk})
assume  $k_i\le j<k_{i+1}$.  Then 
\begin{displaymath}
 \N_{{\alpha},j}=\N_{{\gamma_{i+1}},j}\subs \N_{{\gamma_{i+1}}, j+1}.         
\end{displaymath}

If  $j+1<k_{i+1}$ also then
\begin{displaymath}
 \N_{{\alpha},j+1}=\N_{{\gamma_{i+1}}, j+1}         
\end{displaymath}
so $\N_{{\alpha},j}\subs\N_{{\alpha},j+1} $. If $j+1=k_{i+1}$
then $\N_{{\gamma_{i+1}}, j+1}\subs \N_{{\gamma}_{i+2}, j+1}=\N_{{\alpha},j+1}$, and so $\N_{{\alpha},j}\subs\N_{{\alpha},j+1} $.
In turn, (\ref{Nst-monk}) holds.

\medskip 

To check 
 (\ref{Nst-incr}) let ${\beta}<{\alpha}$. Pick $i$ such that ${\beta}<{\beta}_i$; we would like to show that 
if $k_{\beta,\alpha}= \max\{k_i, k_{{\beta},{\gamma_i}}\}$ works. So suppose that $j\geq k_{\beta,\alpha}$  and  $k_\ell\le j<k_{\ell+1}$ for some $\ell<\oo$. Now  
\begin{displaymath}
 \N_{{\beta},j}\subs \N_{{\gamma_i},j}\subs \N_{{\gamma_{i+1}},j}\subs\dots\subs     \N_{{\gamma}_{\ell+1},j}=
 \N_{{\alpha},j}  
\end{displaymath}
because $j\ge k_{{\beta},{\gamma_i}}$ and $j\ge k_i, k_{i+1},\dots k_\ell$.

\medskip 
To check (\ref{pr:extrahomog})  let ${\gamma}\in C'_{\alpha}$ and recall that ${\beta}_0=\sup C_{\alpha}'<{\alpha}$. 

Assume first that ${\gamma}<{\beta}_0$.
Then $C_{{\beta}_0}=C_{\alpha}\cap {{\beta}_0}$ and so ${\gamma}\in C_{{\beta}_0}'$.
Thus $\N_{{\gamma}+1,j}\subs \N_{{\beta}_0,j}$  by the inductive assumption (\ref{pr:extrahomog}). However $\N_{{\beta}_0,j}\subs \N_{\beta_0+1,j}=\N_{\gamma_0,j}$ by (\ref{Nst-monk}). Thus 
$\N_{{\gamma}+1,j}\subs \N_{{\gamma}_0,j}$ for an arbitrary ${\gamma}\in C_{\alpha}'\cap \beta_0$.

Assume that $k_i\le j<k_{i+1}$. Then $j\ge k_0\dots, k_i$ and so 
\begin{displaymath}
\N_{{\gamma}_0,j}\subs \N_{{\gamma}_1,j}\subs \dots \subs \N_{{\gamma_{i+1}},j}=\N_{{\alpha},j}.  
\end{displaymath}
Hence $\N_{\gamma+1,j}\subs \N_{\alpha,j}$ as required. 

If $\gamma=\beta_0$ then $\N_{\gamma+1,j}=\N_{\gamma_0,j}\subs \N_{{\alpha},j}$ by the above calculation. So ultimately (\ref{pr:extrahomog}) holds.

 \medskip

Property (\ref{pr:manyomegaclosed}) holds because  $\N_{{\alpha},j}=\N_{{\beta}+1,j}$ for some
${\beta}<{\alpha}$ and $[\N_{{\beta}+1,j}]^\oo\subs \N_{{\beta}+1,j} $ by the inductive assumption (\ref{pr:manyomegaclosed}).

\medskip

Property  (\ref{pr:manyallclosed}) holds because
\begin{displaymath}
 \kk_{\alpha}=\bigcup_{j<{\omega}}\N_{{\alpha},j}
\end{displaymath}
and $\N_{{\alpha},j}$ is countably closed by (\ref{pr:manyomegaclosed}).

\medskip

Finally, to check Property (\ref{pr:Ncont}) first note that 
$\kk_{\alpha}\supseteq \bigcup_{{\beta}<{\alpha}}\kk_{\beta}$ by (\ref{Nst-incr}). On the other hand
\begin{displaymath}
 \kk_{\alpha}=\bigcup_{j\in {\omega}}\N_{{\alpha},j}\subs 
 \bigcup_{i,j\in {\omega}}\N_{{\gamma_i},j}=\bigcup_{i\in \oo}\kk_{\gamma_i}\subs 
 \bigcup_{{\beta}<{\alpha}}\kk_{\beta}
\end{displaymath}  by (\ref{eq:Naknotsup}).

This finishes Case 3.

\medskip
\textbf{Case 4: $\alpha$ is a limit and $\sup C'_\alpha=\alpha$.} Now, instead of previous rows, we use the $j^{\textmd{th}}$ column of the already constructed matrix $\<\N_{\zeta,i}:\zeta<\alpha,i<\oo\>$ to provide the new model $\N_{\alpha,j}$. That is, we let

\begin{equation}\label{eq:Naksup}
 \N_{{\alpha},j}=
 \left\{
 \begin{array}{ll}
 \bigcup_{{\beta}\in C_{\alpha}'}\N_{{\beta}+1,j}&\mbox{if $|C_{\alpha}|\le {\mu}_j$, and}\\\\
\N_{0,j}&\mbox{if $|C_{\alpha}|> {\mu}_j$.}
 \end{array}
 \right .
\end{equation}

\medskip
Note that $|C_{\alpha}|\le {\mu}_j$ for almost all $j<\oo$ since $|C_\alpha|<\mu$. So, since $\<\mu_j:j<\oo\>$ is increasing, the set  $\{j<\oo:\N_{\alpha,j}=\N_{0,j}\}$ is a finite initial segment of $\oo$. 

We need to check that properties  (A)--(\ref{pr:last}) are satisfied, so lets start with (\ref{Nst-elem}):  if $\N_{\alpha,j}=\N_{0,j}$ then this is clear so we can assume $|C_{\alpha}|\le {\mu}_j$. This and $|\N_{\beta+1,j}|\leq \mu_j$  implies that $|\N_{\alpha,j}|\leq \mu_j$. Furthermore, $\mu_j+1\subseteq \N_{\beta+1,j}\subseteq \N_{\alpha,j}$ too. 

We will show that $ \N_{\alpha,j}$ is an increasing union of elementary submodels of $H(\theta)$ and so an elementary submodel of $H(\theta)$ itself. If ${\beta}<{\gamma}\in C_{\alpha}'$ then $C_{\gamma}=C_{\alpha}\cap {\gamma}$
and so ${\beta}\in C_{\gamma}'$. Thus $\N_{{\beta}+1,j}\subs \N_{{\gamma},j}\subs \N_{{\gamma}+1,j}$ by 
 the inductive hypothesis (\ref{pr:extrahomog}) and (\ref{Nst-monk}). This finishes the proof of (\ref{Nst-elem}).
 
 \medskip

 Property (\ref{Nst-monk}) is clear from the inductive assumption (\ref{Nst-monk}).  

  \medskip
 
Property (\ref{Nst-incr}): Given ${\beta}<{\alpha}$, we pick ${\gamma}\in C'_{\alpha}\setm (\beta+1)$. If $j\ge k_{{\beta},{\gamma}}$ then $\N_{{\beta},j}\subs\N_{{\gamma},j}\subs \N_{{\gamma}+1,j}\subs \N_{{\alpha},j}$. So  $k_{\beta,\alpha}=k_{\beta,\gamma}$ satisfies the requirements.

\medskip
 
Property (\ref{pr:extrahomog}) holds by definition and Property (\ref{pr:manyparam}) is void in the current case.

\medskip

Property (\ref{pr:manyomegaclosed}): we can assume that $\cf({\alpha})>{\omega}$ otherwise there is nothing to prove. Consider  an arbitrary   $a\in \br \N_{{\alpha},j};{\omega};$. Then there is $I\in \br C_{\alpha}';{\omega};$
such that   $a\subs \bigcup_{{\beta}\in I}\N_{{\beta}+1,j}$. 

Pick  ${\gamma}\in C_{\alpha}'\setm (\sup I+1)$ and note that $I\subseteq C'_\gamma$ as well (since $C_\gamma=\gamma\cap C_\alpha$). In turn, $\bigcup_{{\beta}\in I}\N_{{\beta}+1,j}\subs \N_{\gamma,j}$ by (\ref{pr:extrahomog}); so $a\subs \N_{\gamma,j}\subseteq \N_{\gamma+1,j}$. $\N_{\gamma+1,j}$ is countably closed so $a\in \N_{\gamma+1,j}$. Finally, $\N_{\gamma+1,j}\subs \N_{\alpha,j}$ by the definition of $\N_{\alpha,j}$ so $a\in \N_{\alpha,j}$ as desired.

\medskip

Property (\ref{pr:manyallclosed}): we assume that $\cf(\alpha)=\oo$ otherwise 
(\ref{pr:manyomegaclosed}) implies that $\N_{{\alpha},j}$ are countably closed.

Pick a sequence of  ordinals  ${\beta}_0<{\beta}_1<\dots$ in  $C_{\alpha}'$ which is cofinal in $\alpha$ and let $\gamma_i=\beta_i+1$. Fix  a strictly increasing sequence of natural numbers $k_0<k_1<\dots$ as well such that 
\begin{equation}
 \N_{{\gamma}_i,j}\subs \N_{{\gamma}_{i+1},j} \textmd{ for all }j\ge k_i. 
\end{equation}

Let 
\begin{displaymath}
\kkk_{{\alpha},j}=\left\{
\begin{array}{ll}
\N_{{\gamma}_0,j}&\mbox{if $j<k_0,$ and}\\\\ 
\N_{{\gamma}_{i+1},j}&\mbox{if $k_i\le j<k_{i+1}$.}
\end{array}
\right . 
\end{displaymath}


To prove (\ref{pr:manyallclosed}) we will show  that  $\kk_{\alpha}=\bigcup_{j<{\omega}}\kkk_{{\alpha},j}$.  First, note that $(\kkk_{\alpha,j})_{j<\oo}$ is increasing, and $\kkk_{\alpha,j}\subseteq \N_{\alpha,j}\subseteq \kk_\alpha$ for almost all $j<\oo$ (whenever $|C_\alpha|\leq \mu_j$). Hence  $\bigcup_{j<{\omega}}\kkk_{{\alpha},j}\subseteq \kk_\alpha$.


Now, assume that $y\in \kk_{{\alpha}}$; then $y\in \N_{{\alpha},j}$ for some $j<\oo$
and so $y\in \N_{{\beta}+1,j}$ for some ${\beta}\in C'_{\alpha}$. Pick $i<\oo$ such that ${\beta}+1<{\gamma}_i$ and let $\ell=\max\{k_i, k_{{\beta}+1,{\gamma}_i},j\}+1$. We would like to show
\begin{claim}
 $y\in \kkk_{\alpha,\ell}$.
\end{claim}
\begin{proof}First, note that
\begin{displaymath}
 y\in \N_{{\beta}+1,j}\subs \N_{{\beta}+1,\ell}\subs \N_{{\gamma}_i,\ell}
\end{displaymath} 
as $\ell>k_{{\beta}+1,{\gamma}_i},j$. 
Also, there is an $i'<\oo$ (at least $i$) such that  $k_{i'}\le \ell<k_{{i'}+1}$ and 
then   $\kkk_{{\alpha},\ell}=\N_{{\gamma}_{{i'}+1},\ell}$. 
Finally, $\ell\ge k_i, k_{i+1},\dots k_{i'}$ implies that 
\begin{displaymath}
 \N_{{\gamma}_i,\ell}\subs \N_{{\gamma}_{i+1},\ell }\subs \dots\subs     
 \N_{{\gamma}_{{i'}+1},\ell}=\kkk_{{\alpha},\ell}.  
\end{displaymath}

Thus $y\in \kkk_{{\alpha},\ell}$.
\end{proof}

So we verified property  (\ref{pr:manyallclosed}).

\medskip

Property (\ref{pr:Ncont}): first, $\kk_{\alpha}\supseteq \bigcup_{{\beta}<{\alpha}}\kk_{\beta}$ 
immediately follows from (\ref{Nst-incr}). 

On the other hand, 
\begin{displaymath}
 \kk_{\alpha}=\bigcup_{j< {\omega}}\N_{{\alpha},j}=
 \bigcup_{j< {\omega}}\bigcup_{\beta\in C_\alpha'}\N_{{\beta}+1,j}= 
 \bigcup_{\beta\in C_\alpha'}\bigcup_{j< {\omega}}\N_{{\beta}+1,j}=
 \bigcup_{\beta\in C_\alpha'}\kk_{\beta+1}\subs \bigcup_{{\beta}<{\alpha}}\kk_{\beta}
\end{displaymath} by (\ref{eq:Naksup}).

\medskip

This concludes Case 4 and hence the inductive construction of the matrix $\<\N_{\alpha,j}:\alpha<\kappa,j<\oo\>$. 
In the turn, we constructed the desired sage model in Case II as well; this finishes the proof of Theorem \ref{tm:niceomegadavies}.
\end{proof}

\medskip
 
Let us remark that if one only aims to construct a high Davies-tree (which is not necessarily sage) then weaker assumptions suffice. We say that $\Box^{***}_{{\omega}_1,\mu}$ holds if there is a sequence $(C_\alpha)_{\alpha<\mu^+}$ and a club $D\subseteq\mu^+$ such that  for every $\alpha\in D$ with $\cf(\alpha)\geq \omg$ the following holds:

\begin{enumerate}
 \item[(v1)] $C_\alpha\subseteq\alpha$ and $C_\alpha$ is unbounded in 
$\alpha$;
\item[(v2)] for all $V\in[C_\alpha]^{{\omega}}$ there is 
$\beta<\alpha$ such that  $V\subseteq C_\beta\in [C_\alpha]^{{\omega}}$.
\end{enumerate}

We only state the following theorem without proof:

\begin{theorem}\label{thm:onlyhigh} There is a high Davies-tree for ${\kappa}$ over $x$ whenever
\begin{enumerate}
 \item\label{neccond} ${\kappa}={\kappa}^{\omega}$, and
 \item\label{mupluss} ${{\mu}}^{\omega}={\mu}^+$ and $\Box^{***}_{{\omega}_1,\mu}$ holds for all  $\mu$ with $\mf c< {\mu}<{\kappa}$ and  $\cf({\mu})={\omega}$.
 \end{enumerate} 
\end{theorem}

\section{Final thoughts and acknowledgements}

We hope that the reader found the above results and proofs as entertaining and instrumental as we did when first working through these diverse topics. Our goal with this paper was to demonstrate that Davies-trees and high Davies-trees can provide a general framework for solving combinatorial problems by the \emph{most straightforward approach}: list your objectives and then meet these goals one by one inductively. In places where originally cumbersome inductions were applied, Davies-trees and high Davies-trees allow a reduction to the simplest cases i.e. one needs to deal with countable or size $\mf c$ approximations of the final structure even when the goal is to construct something of much bigger size. It is our belief that this technique will find its well deserved place in many more proofs in the future.  

\medskip

Part of this project was carried out at the University of Toronto during the first author's graduate studies. We thank D. Milovich, W. Weiss and the Toronto set theory group for fruitful discussions and several useful comments. Let us thank N. Barton, R. Carroy and S. Friedman for further advice on improving our exposition. Finally, we are deeply thankful for the anonymous referee for his/her careful reading and many insightful advice.

\medskip

The first author was supported in part by the Ontario Trillium Scholarship and the FWF Grant I1921. The preparation of this paper was also supported by NKFIH grant no. K113047.

  \bibliographystyle{asl}

\end{document}